\documentclass{amsart}

\begin{document}

\newcommand{\nc}{\newcommand}

\nc{\pr}{\noindent{\em Proof. }} \nc{\g}{\mathfrak g}
\nc{\n}{\mathfrak n} \nc{\opn}{\overline{\n}}\nc{\h}{\mathfrak h}
\renewcommand{\b}{\mathfrak b}
\nc{\Ug}{U(\g)} \nc{\Uh}{U(\h)} \nc{\Un}{U(\n)}
\nc{\Uopn}{U(\opn)}\nc{\Ub}{U(\b)} \nc{\p}{\mathfrak p}
\renewcommand{\l}{\mathfrak l}
\nc{\z}{\mathfrak z} \renewcommand{\h}{\mathfrak h}
\nc{\m}{\mathfrak m}
\renewcommand{\k}{\mathfrak k}
\nc{\opk}{\overline{\k}}
\nc{\opb}{\overline{\b}}

\newtheorem{theorem}{Theorem}{}
\newtheorem{lemma}[theorem]{Lemma}{}
\newtheorem{corollary}[theorem]{Corollary}{}
\newtheorem{conjecture}[theorem]{Conjecture}{}
\newtheorem{proposition}[theorem]{Proposition}{}
\newtheorem{axiom}{Axiom}{}
\newtheorem{remark}{Remark}{}
\newtheorem{example}{Example}{}
\newtheorem{exercise}{Exercise}{}
\newtheorem{definition}{Definition}{}

\title{The Poisson geometry of the conjugation quotient map for simple algebraic groups and deformed Poisson W--algebras}

\author[A. Sevostyanov]{A. Sevostyanov}
\address{Institute of Pure and Applied Mathematics \\
University of Aberdeen \\ Aberdeen AB24 3UE \\ United Kingdom }
\email{a.sevastyanov@abdn.ac.uk}


\begin{abstract}
We define Poisson structures on certain transversal slices to conjugacy classes in complex simple algebraic groups. These
slices are associated to the elements of the Weyl group, and
the Poisson structures on them are analogous to the Poisson structures introduced in papers \cite{DB,Pr} on the Slodowy slices in complex simple Lie algebras.
The quantum deformations of these Poisson structures are known as W--algebras of finite type. As an application of our definition we obtain
some new Poisson structures on the coordinate rings of simple Kleinian singularities.
\end{abstract}

\maketitle

\pagestyle{myheadings}

\markboth{A. SEVOSTYANOV}{CONJUGATION QUOTIENT MAP}

\section*{Introduction}

In 1970 Brieskorn conjectured that simple singularities can be obtained as intersections of the nilpotent cone of a complex simple Lie algebra $\g$ with transversal slices to adjoint orbits of subregular nilpotent elements in $\g$ (see \cite{Br}). This conjecture was proved by Slodowy in book \cite{SL} where for each nilpotent element $e$ a suitable transversal slice $s(e)$ to the adjoint orbit of $e$ was constructed.

The simple singularities appear as some singularities of the fibers of the adjoint quotient map $\delta_{\g}: \g \rightarrow \h/W$ generated by the inclusion $\mathbb{C}[\h]^W\simeq
\mathbb{C}[\g]^G\hookrightarrow \mathbb{C}[\g]$, where $\h$ is a Cartan subalgebra of $\g$ and $W$ is the Weyl group of the pair $(\g, \h)$. The fibers of
the adjoint quotient map are unions of adjoint orbits in $\g$. Each fiber of
$\delta_{\g}$ contains a single orbit which consists of regular
elements. The singularities of the fibers correspond to irregular
elements. For irregular $e$ the restriction of the adjoint quotient map $\delta_{\g}$ to the slice $s(e)$ has some singular fibers, and $s(e)$ can be regarded as a deformation of these singularities.

In papers \cite{DB,Pr} using Hamiltonian reduction it was shown that the slices $s(e)$ can be naturally equipped with Poisson structures and that these Poisson structures can be quantized. This yields noncommutative deformations of the singularities of the fibers of the adjoint quotient map. These noncommutative deformations called W--algebras of finite type play an important role in the classification of the generalized Gelfand--Graev representations of the Lie algebra $\g$ \cite{Ka,Pr}.

In this paper we are going to outline a similar construction for algebraic groups. Let $G$ be a complex simple algebraic group with Lie algebra $\g$. In case of algebraic groups, instead of the adjoint quotient map, one should consider the conjugation quotient map $\delta_G: G \rightarrow T/W$ generated by the inclusion $\mathbb{C}[T]^W\simeq
\mathbb{C}[G]^G\hookrightarrow \mathbb{C}[G]$, where $T$ is the maximal torus of $G$ corresponding to the Cartan subalgebra $\h$ and $W$ is the Weyl group of the pair $(G, T)$. Some fibers of this map are singular and one can study these singularities by restricting $\delta_G$ to certain transversal slices to conjugacy classes in $G$. These slices were defined in paper \cite{S1}, and they are associated to pairs $(\p,s)$, where $\p$ is a parabolic subalgebra in $\g$ and $s$ is an element of the Weyl group $W$.

To describe the Poisson structure on the slice $s(e)$ by means of Hamiltonian reduction we first recall that the space $\g^*$ dual to $\g$ carries the standard Kirillov-Kostant Poisson bracket. The coadjoint action of the adjoint group $G'$ of $\g$ on $\g^*$ is Hamiltonian, i.e. this action preserves the Poisson bracket on $\g^*$. One can restrict this action to a unipotent subgroup $N\subset G'$ associated to the nilpotent element $e$ as follows.

Let $(e,h,f)$ be an $\mathfrak{sl}_2$--triple associated to $e$, i.e.
elements $f,h\in \g$ obey the following commutation relations $[h,e]=2e$, $[h,f]=-2f$, $[e,f]=h$.
Under the action of ${\rm ad}~h$
we have a decomposition
\begin{equation}\label{gr**}
\g=\oplus_{i\in \mathbb{Z}}\g(i),~{\rm where}~ \g(i)=\{x\in \g
\mid [h,x]=ix\}.
\end{equation}
Denote by $\chi$ the element of
$\g^*$ which corresponds to $e$ under the isomorphism $\g\simeq
\g^*$ induced by the Killing form. The skew--symmetric bilinear form $\omega$ on $\g(-1)$ defined by
$\omega(x,y)=\chi([x,y])$ is nondegenerate. Fix an isotropic Lagrangian
subspace $l$ of $\g(-1)$ with respect to $\omega$.
Let
$$
\n
={l}\oplus \bigoplus_{i\leq -2}\g(i),
$$
and $N$ the Lie subgroup of $G'$ which corresponds to the Lie subalgebra $\n\subset \g$.

The restriction of the coadjoint action of $G'$ to $N$ is still Hamiltonian, and hence the quotient $\g^*/N$ naturally acquires a Poisson structure. We denote by $\mu:\g^* \rightarrow \n^*$ the moment map for this action. Note that the quotient $\g^*/N$ is not smooth, and we understand the Poisson structure on $\g^*/N$ in the sense that the set $C^\infty(\g^*)^N$ of $N$--invariant functions on $\g^*$ is a Poisson subalgebra in $C^\infty(\g^*)$.

In papers \cite{DB,GG,K} it was proved that the reduced Poisson manifold $\mu^{-1}(\chi|_\n)/N$ corresponding to the value $\chi|_\n$ of the moment
map can be identified with $s(e)$ under the isomorphism $\g^*\simeq \g$ induced by the Killing form. This result is known as
the Kostant cross--section theorem. By this theorem the slice $s(e)$ naturally acquires a Poisson structure, and $s(e)$ becomes a Poisson
submanifold of the quotient $\g^*/N$. As a variety the slice $s(e)$ can be identified with $e+\z(f)$, where $\z(f)$ is the centralizer of $f$ in $\g$.

Now we briefly describe the main construction of this paper. As we already mentioned above algebraic group counterparts of the Slodowy slices $s(e)$ were introduced in paper \cite{S1}. We shall equip these transversal slices to conjugacy classes in the group $G$ with some Poisson structures. These slices are of the form $N_sZs^{-1}$ where $s\in G$ is a representative of an element of the Weyl group $W$, $N_s=\{n\in N\mid sns^{-1}\in \overline P\}$, $N$ is the unipotent radical of a parabolic subgroup $P$ of $G$ with Levi factor $L$, $\overline P$ is the opposite parabolic subgroup, and $Z=\{z\in L\mid szs^{-1}=z\}$ is the centralizer of $s$ in $L$. If the action of $s$ on $N$ by conjugations has no fixed points then $N_sZs^{-1}$ is indeed a transversal slice to conjugacy classes in $G$ (for a precise statement see \cite{S1} or Proposition \ref{prop1} below). Moreover, the quotient $NZs^{-1}N/N$ with respect to the action of $N$ on $NZs^{-1}N$ by conjugations is isomorphic to $N_sZs^{-1}$ (see \cite{S1} or Proposition \ref{prop2} below), and hence $N_sZs^{-1}$ is a subvariety of $G/N$.

Now in order to equip the slice  $N_sZs^{-1}$ with a Poisson structure we recall that the action of $G$ on itself by conjugations can be naturally put into the context of Poisson geometry. Namely, equip the group $G$ with the structure of a quasitriangular Poisson-Lie group, i.e. with a Poisson bracket on $G$ such that the product map $G\times G\rightarrow G$ is a Poisson mapping. (A recollection of the results on Poisson--Lie groups and Poisson geometry used in this paper can be found in Section \ref{plg}). If we denote by $(\g,\g^*)$ the tangent Lie bialgebra of $G$ then one can define the dual Poisson--Lie group $G^*$ with the tangent Lie bialgebra $(\g^*,\g)$. As a manifold the dual Poisson--Lie group $G^*$ is isomorphic to a dense open subset in $G$, and there is a Poisson group action, i.e. a Poisson map $G\times G^* \rightarrow G^*$, called the dressing action. If we realize $G^*$ as the dense open subset in $G$ then the dressing action is induced by the action of $G$ on itself by conjugations. Denote by $G_*$ the group $G$ equipped with the Poisson structure induced from $G^*$. In Lemma \ref{radmiss} and Theorem \ref{mainth} it is proved that under some compatibility condition for $N$, $s$ and the Poisson--Lie group structure on $G$ the subgroup $N\subset G$ is admissible, in the sense that $G_*/N$ is naturally equipped with the Poisson structure induced from $G_*$, and that the quotient $N_sZs^{-1}\simeq NZs^{-1}N/N$ is a Poisson submanifold in $G_*/N$. Here as in case of Lie algebras the quotient $G_*/N$ is singular, and the Poisson structure on $G_*/N$ should be understood in the sense that the set $C^\infty(G_*)^N$ of $N$--invariant functions on $G_*$ is a Poisson subalgebra in $C^\infty(G_*)$

The definition of the Poisson structure on the slice $N_sZs^{-1}$ is similar to the definition of the Poisson structure on the Slodowy slice
$s(e)$ discussed above, the variety $NZs^{-1}N$ being a counterpart of the level surface $\mu^{-1}(\chi|_\n)$ of the moment map in case of Lie
algebras. We call the Poisson algebra of regular functions on the slice $N_sZs^{-1}$ a deformed Poisson W--algebra. The deformed Poisson W--algebras
can be naturally quantized in the framework of quantum group theory. This will be explained in a subsequent paper.

In case when $s$ is a representative of a Coxeter element in $W$ and $P$ is a Borel subgroup of $G$ the slice $N_ss^{-1}$ was introduced by R. Steinberg in \cite{St2} (in this case the subgroup $Z$ is trivial) and the Poisson structure on the slice $N_ss^{-1}$ was defined in \cite{STSh}. Actually this Poisson structure is trivial, and the quantization of the Poisson algebra of regular functions on $N_ss^{-1}$ is isomorphic to the center of the quantized algebra of regular functions on the Poisson--Lie group $G^*$ (see \cite{S2}). This fact is of primary importance for classification of the Whittaker representations of the quantum group $U_q(\g)$ \cite{S3}.

Note that the $G$--invariant functions on $G_*$ lie in the center of the Poisson algebra of functions on $G_*$ \cite{dual}. Therefore the intersections of the slices $N_sZs^{-1}$ with the fibers of the conjugation quotient map $\delta_G:G\rightarrow T/W$ are also equipped with Poisson structures. In particular, if such a fiber is singular our construction yields a Poisson structure related to the corresponding singularity. In Section \ref{sl3} we study in detail the case when $s$ is subregular in $G$ and $\p$ is a parabolic subalgebra in $\g$. It is shown in \cite{S1} that in this case the intersections of the slice $N_sZs^{-1}$ with the fibers of the conjugation quotient map may only have simple isolated singularities, and we show that using Theorem \ref{mainth} one can equip the slice $N_sZs^{-1}$ and the coordinate rings of the corresponding simple singularities with Poisson structures.

In the simplest case when $\g={\mathfrak s \mathfrak l}_3$, $s$ is a representative of the longest element of the Weyl group and $\p$ is a Borel subalgebra one can calculate the corresponding Poisson structures explicitly. We show that in this case the singular fiber $\delta_G^{-1}(1)$ of the map $\delta_G: N_sZs^{-1}\rightarrow T/W$ has simple $A_2$--type singularity, and the corresponding Poisson structure on the singular fiber is proportional to that obtained in \cite{Pr} for the singular fiber $\delta_{\g}^{-1}(0)$ of the adjoint quotient map $\delta_{\g}: s(e)\rightarrow \h/W$ in case of a subregular nilpotent element $e\in \g$.

In general the Poisson structures on the slices $N_sZs^{-1}$ are difficult to describe explicitly. In Section \ref{long} we consider the case when $s$ is the reflection with respect to a long root $\beta$ and $\p=\g(0)\oplus \g(1)\oplus \g(2)$ where the components $\g(i)$ are defined by formula (\ref{gr**}) where $e$ is a root vector corresponding to $\beta$. In this case the Poisson structure on the slice $N_sZs^{-1}$ can be described explicitly. An analogous Poisson structure on the slice $s(e)$ for a long root vector $e$ in $\g$ was considered in \cite{Pr1}. The quantization of this Poisson structure is related to the Joseph ideal of the universal enveloping algebra $U(\g)$.

This paper was completed during
my stay in Max--Planck--Institut f\"{u}r Mathematik, Bonn in April--May 2009. I would like to thank Max--Planck--Institut
f\"{u}r Mathematik, Bonn for hospitality.


\section{Transversal slices to conjugacy classes in algebraic
groups} \label{slices1}

In this section following \cite{S1} we recall the definition of the algebraic group counterparts of the
Slodowy slices and an analogue of the Kostant cross--section
theorem for them.

Let $G$ be a complex semisimple (connected) algebraic group, $\g$ its Lie
algebra, $H$ a maximal torus in $G$. Denote by $\h$ the Cartan subalgebra in $\g$ corresponding to $H$. Let $\Delta$ be the root system of the pair $(\g,\h)$. For any root $\alpha\in \Delta$ we denote by $\alpha^\vee\in \h$ the corresponding coroot.

Let $s$ be an element of the Weyl group $W$ of the pair $(\g,\h)$ and $\h_{\mathbb{R}}$ the real form of $\h$, the real linear span of simple coroots in $\h$. The set of roots $\Delta$ is a subset of the dual space $\h_\mathbb{R}^*$.

The Weyl group element $s$ naturally acts on $\h_{\mathbb{R}}$ as an orthogonal transformation with respect to the scalar product induced by the Killing form of $\g$. Using the spectral theory of orthogonal transformations we can decompose $\h_{\mathbb{R}}$ into a direct orthogonal sum of $s$--invariant subspaces,
\begin{equation}\label{hdec}
\h_\mathbb{R}=\bigoplus_{i=0}^{K} \h_i,
\end{equation}
where we assume that $\h_0$ is the linear subspace of $\h_{\mathbb{R}}$ fixed by the action of $s$, and each of the other subspaces $\h_i\subset \h_\mathbb{R}$, $i=1,\ldots, K$, is either two--dimensional or one--dimensional and the Weyl group element $s$ acts on it as rotation with angle $\theta_i$, $0<\theta_i<\pi$ or as the reflection with respect to the origin, respectively. Note that since $s$ has finite order $\theta_i=\frac{2\pi}{m_i}$, $m_i\in \mathbb{N}$.

Since the number of roots in the root system $\Delta$ is finite one can always choose elements $h_i\in \h_i$, $i=0,\ldots, K$, such that $h_i(\alpha)\neq 0$ for any root $\alpha \in \Delta$ which is not orthogonal to the $s$--invariant subspace $\h_i$ with respect to the natural pairing between $\h_{\mathbb{R}}$ and $\h_{\mathbb{R}}^*$.

Now we consider certain $s$--invariant subsets of roots $\overline{\Delta}_i$, $i=0,\ldots, K$, defined as follows
\begin{equation}\label{di}
{\overline{\Delta}}_i=\{ \alpha\in \Delta: h_j(\alpha)=0, j>i,~h_i(\alpha)\neq 0 \},
\end{equation}
where we formally assume that $h_{K+1}=0$.
Note that for some indexes $i$ the subsets ${\overline{\Delta}}_i$ are empty, and that the definition of these subsets depends on the order of terms in direct sum (\ref{hdec}).

We also define other $s$--invariant subsets of roots ${\Delta}_{i_k}$
for all indexes $i_k$, $k=0,\ldots, M$ such that $\overline{\Delta}_{i_k}$ is not empty,
\begin{equation}\label{dik}
{\Delta}_{i_k}=\bigcup_{i_j\leq i_k}\overline{\Delta}_{i_j}.
\end{equation}
For convenience we assume that indexes $i_k$ are labeled in such a way that $i_j<i_k$ if and only if $j<k$.
According to this definition we have a chain of strict inclusions
\begin{equation}\label{inc}
\Delta_{i_M}\supset\Delta_{i_{M-1}}\supset\ldots\supset\Delta_{i_0},
\end{equation}
such that $\Delta_{i_M}=\Delta$, $\Delta_{0}=\{\alpha \in \Delta: s\alpha=\alpha\}$ is the set of roots fixed by the action of $s$, and ${\Delta}_{i_k}\setminus {\Delta}_{i_{k-1}}=\overline{\Delta}_{i_k}$. Observe also that the root system $\Delta$ is the disjoint union of the subsets $\overline{\Delta}_{i_k}$,
$$
\Delta=\bigcup_{k=0}^{M}\overline{\Delta}_{i_k}.
$$

Now assume that
\begin{equation}\label{cond}
|h_{i_k}(\alpha)|>|\sum_{l\leq j<k}h_{i_j}(\alpha)|, ~{\rm for~any}~\alpha\in \overline{\Delta}_{i_k},~k=0,\ldots, M,~l<k.
\end{equation}
Condition (\ref{cond}) can be always fulfilled by suitable rescalings of the elements $h_{i_k}$.

Consider the element
$$
\bar{h}=\sum_{k=0}^{M}h_{i_k}\in \h_\mathbb{R}.
$$
From definition (\ref{di}) of the sets $\overline{\Delta}_i$ we obtain that for $\alpha \in \overline{\Delta}_{i_k}$
\begin{equation}\label{dech}
\bar{h}(\alpha)=\sum_{j\leq k}h_{i_j}(\alpha)=h_{i_k}(\alpha)+\sum_{j< k}h_{i_j}(\alpha)
\end{equation}
Now condition (\ref{cond}), the previous identity and the inequality $|x+y|\geq ||x|-|y||$ imply that for $\alpha \in \overline{\Delta}_{i_k}$ we have
$$
|\bar{h}(\alpha)|\geq ||h_{i_k}(\alpha)|-|\sum_{j< k}h_{i_j}(\alpha)||>0.
$$
Since $\Delta$ is the disjoint union of the subsets $\overline{\Delta}_{i_k}$, $\Delta=\bigcup_{k=0}^{M}\overline{\Delta}_{i_k}$, the last inequality ensures that  $\bar{h}$ belongs to a Weyl chamber of the root system $\Delta$, and one can define the subset of positive roots $\Delta_+$ and the set of simple positive roots $\Gamma$ with respect to that chamber. From condition (\ref{cond}) and formula (\ref{dech}) we also obtain that a root $\alpha \in \overline{\Delta}_{i_k}$ is positive if and only if $h_{i_k}(\alpha)>0$.

To define the algebraic group analogues of the Slodowy slices we shall also need a parabolic subalgebra $\p$ of $\g$ associated to the semisimple element $\bar{h}_0=\sum_{k=0,i_k>0}^{M}h_{i_k}\in \h_\mathbb{R}$ associated to $s\in W$. This subalgebra is defined with the help of the linear eigenspace decomposition of $\g$ with respect to the adjoint action of $\bar{h}_0$ on $\g$, $\g=\bigoplus_{m}(\g)_m$, $(\g)_m=\{ x\in \g \mid [\bar{h}_0,x]=mx\}$, $m \in \mathbb{R}$. By definition $\p=\bigoplus_{m\leq 0}(\g)_m$ is a parabolic subalgebra in $\g$, $\n=\bigoplus_{m<0}(\g)_m$ and $\l=\{x\in \g \mid [\bar{h}_0,x]=0\}$ are the nilradical and the Levi factor of $\p$, respectively. We denote by $P$ the corresponding parabolic subgroup of $G$, by $N$ the unipotent radical of $P$ and by $L$ the Levi factor of $P$. The subgroups of $P$, $N$ and $L$ have Lie algebras $\p$, $\n$ and $\l$, respectively, and both $P$ and $L$ are connected. Note that we have natural inclusions of Lie algebras $\p\supset\b\supset\n$, where $\b$ is the Borel subalgebra of $\g$ corresponding to the system $-\Gamma$ of simple roots, and $\Delta_{0}$ is the root system of the reductive Lie algebra $\l$.

Let $X_\alpha\subset \g$ be the root subspace of $\g$ corresponding to root $\alpha \in \Delta$.
Fix a system of root vectors $e_\alpha\in X_{\alpha}, \alpha \in \Delta$ such that if $[e_{\alpha},e_\beta]=N_{\alpha,\beta}e_{\alpha+\beta}\in X_{\alpha+\beta}$ for any pair $\alpha,\beta\in \Gamma$ of simple positive roots then $[e_{-\alpha},e_{-\beta}]=N_{\alpha,\beta}e_{-\alpha-\beta}\in X_{-\alpha-\beta}$.

Recall that by Theorem 5.4.2. in \cite{GG1} one can uniquely choose a representative $s\in G$ for the Weyl group element $s\in W$ in such a way that the operator ${\rm Ad}s$ sends root vectors $e_{\pm \alpha}\in X_{\pm \alpha}$ to $e_{\pm s\alpha}\in X_{\pm s\alpha}$ for any simple positive root $\alpha \in \Gamma$.
We denote this representative by the same letter, $s\in G$. The representative $s\in G$ is called the normal representative of the Weyl group element $s\in W$. If the order of the Weyl group element $s\in W$ is equal to $R$ then the inner automorphism ${\rm Ad}s$ of the Lie algebra $\g$ has order at most $2R$, ${\rm Ad}s^{2R}={\rm id}$. We also recall that the operator ${\rm Ad}s$ sends each root subspace $X_\alpha\subset \g$, $\alpha \in \Delta$  to $X_{s\alpha}$.

The element $s\in G$ naturally acts on $G$
by conjugations. Let $Z$ be the set of $s$-fixed points in $L$,
\begin{equation}\label{defz}
Z=\{z\in L\mid szs^{-1}=z\},
\end{equation}
and
\begin{equation}\label{defns}
N_s=\{n\in N\mid sns^{-1}\in \overline N\},
\end{equation}
where $\overline N$ is the unipotent radical of the parabolic subgroup $\overline P\subset G$ opposite to $P$. Note that ${\rm dim}~N_s=l(s)$, where $l(s)$ is the length of the Weyl group element $s\in W$ with respect to the system $\Gamma$ of simple roots.
Clearly, $Z$ and $N_s$ are subgroups in $G$, and $Z$ normalizes
both $N$ and $N_s$. Denote by $\n_s$ and $\z$ the Lie algebras of
$N_s$ and $Z$, respectively.

Note that, since the operator ${\rm Ad}s$ sends root vectors $e_{\pm \alpha}\in X_{\pm \alpha}$ to $e_{\pm s\alpha}\in X_{\pm s\alpha}$ for any simple positive root $\alpha \in \Gamma$ and the root system of the reductive Lie algebra $\l$ is fixed by the action of $s$, the semisimple part $\m$ of the Levi subalgebra $\l$ is fixed by the action of ${\rm Ad}s$. In fact in this case $\z=\m\oplus \h_z$ and $\z \cap
\h= \mathbb{C}\h_0$, where $\h_z$ is a Lie subalgebra of the center of $\l$ and $\mathbb{C}\h_0$ is the linear subspace of $\h$ fixed by the action of $s$.

Now consider the subvariety $N_sZs^{-1}\subset G$. The variety $N_sZs^{-1}$ turns out to be transversal to the set of conjugacy classes in $G$. This is also related to
the following statement which is an
analogue of the Kostant cross--section theorem for the subvariety
$N_sZs^{-1}\subset G$.
\begin{proposition}\label{prop2}{\bf (\cite{S1}, Proposition 1)}
Let $s\in W$ be an element of the Weyl group $W$ of the pair $(\g,\h)$.
Let $\overline{h}_0\in \h_\mathbb{R}$ be a semisimple element associated to $s$, $\overline{h}_0= \sum_{k=0,i_k>0}^{M}h_{i_k}$, where elements $h_{i_k}\in \h_{i_k}$ satisfy conditions (\ref{cond}) and $\h_{i_k}\subset \h_\mathbb{R}$ are the subspaces of $\h_\mathbb{R}$ defined in (\ref{hdec}). Let $\n\subset \g$ be the nilradical of the parabolic subalgebra $\p$ defined with the help of $\overline{h}_0$, $\n=\bigoplus_{m<0}(\g)_m$,$(\g)_m=\{ x\in \g \mid [\bar{h}_0,x]=mx\}$, $m \in \mathbb{R}$ and $\l=(\g)_0$ the Levi factor of $\p$. Denote by $P$, $N$ and $L$ the corresponding subgroups of $G$ and by $s\in G$ the normal representative of the Weyl group element $s\in W$. Let $Z$ be the centralizer of $s$ in $L$,
$$
Z=\{z\in L\mid szs^{-1}=z\},
$$
and
$$
N_s=\{n\in N\mid sns^{-1}\in \overline N\},
$$
where $\overline N$ is the unipotent radical of the parabolic subgroup $\overline P\subset G$ opposite to $P$.
Then the conjugation map
\begin{equation}\label{cross}
\alpha: N\times N_sZs^{-1}\rightarrow NZs^{-1}N
\end{equation}
is an isomorphism of varieties.
\end{proposition}

\begin{remark}
The particular choice of the normal representative $s\in G$ for a Weyl group element $s\in W$ in the previous proposition is technically convenient but actually not important. Actually any representative of $s\in W$ in the normalizer of $\h$ in $G$ is $H$--conjugate to an element from $Zs$, where $s\in G$ is the normal representative of $s\in W$.
\end{remark}

\begin{proposition}\label{prop1}{\bf (\cite{S1}, Proposition 2)}
Under the conditions of Proposition \ref{prop2} the variety $N_sZs^{-1}\subset G$ is a transversal
slice to the set of conjugacy classes in $G$.
\end{proposition}


\section{Poisson--Lie groups and Poisson reduction}
\label{plg}

\setcounter{equation}{0}

In this section we recall some results related to Poisson--Lie groups and Poisson geometry (see
\cite{ChP,Dm,fact,dual}). These results will be used in the next section to
equip the slices $N_sZs^{-1}$ defined in Section \ref{slices1} with Poisson structures.

Let $G$ be a
finite--dimensional Lie group equipped with a Poisson bracket,
$\frak g$ its Lie algebra. $G$ is called a Poisson--Lie group if
the multiplication $G\times G \rightarrow G$ is a Poisson map. A
Poisson bracket satisfying this axiom is degenerate and, in
particular, is identically zero at the unit element of the group.
Linearizing this bracket at the unit element defines the structure
of a Lie algebra in the space $T^*_eG\simeq {\frak g}^*$. The pair
(${\frak g},{\frak g}^{*})$ is called the tangent bialgebra of
$G$.

Lie brackets in $\frak{g}$ and $\frak{g}^{*}$ satisfy the
following compatibility condition:

{\em Let }$\delta: {\frak g}\rightarrow {\frak g}\wedge {\frak g}$
{\em be the dual  of the commutator map } $[,]_{*}: {\frak
g}^{*}\wedge {\frak g}^{*}\rightarrow {\frak g}^{*}$. {\em Then }
$\delta$ {\em is a 1-cocycle on} $  {\frak g}$ {\em (with respect
to the adjoint action of } $\frak g$ {\em on} ${\frak
g}\wedge{\frak g}$).

Let $c_{ij}^{k}$ and $f^{ab}_{c}$ be the structure constants of ${\frak
g}$ and ${\frak g}^{*}$, respectively, with respect to the dual bases $\{e_{i}\}$ and
$\{e^{i}\}$ in ${\frak g}$ and ${\frak g}^{*}$. The compatibility
condition means that

$$
c_{ab}^{s} f^{ik}_{s} ~-~ c_{as}^{i} f^{sk}_{b} ~+~ c_{as}^{k}
f^{si}_{b} ~-~ c_{bs}^{k} f^{si}_{a} ~+~ c_{bs}^{i} f^{sk}_{a} ~~=
~~0.
$$
This condition is symmetric with respect to exchange of $c$ and
$f$. Thus if $({\frak g},{\frak g}^{*})$ is a Lie bialgebra, then
$({\frak g}^{*}, {\frak g})$ is also a Lie bialgebra.

The following proposition shows that the category of
finite--dimensional Lie bialgebras is isomorphic to the category
of finite--dimensional connected simply connected Poisson--Lie
groups.
\begin{proposition}{\bf (\cite{ChP}, Theorem 1.3.2)}
If $G$ is a connected simply connected finite--dimensional Lie
group, every bialgebra structure on $\frak g$ is the tangent
bialgebra of a unique Poisson structure on $G$ which makes $G$
into a Poisson--Lie group.
\end{proposition}

Let $G$ be a finite--dimensional Poisson--Lie group, $({\frak
g},{\frak g}^{*})$ the tangent bialgebra of $G$. The connected
simply connected finite--dimensional Poisson--Lie group
corresponding to the Lie bialgebra $({\frak g}^{*}, {\frak g})$ is
called the dual Poisson--Lie group and denoted by $G^*$.

In this paper we shall need a special class of factorizable Lie bialgebras.
This class  is slightly smaller than the class of quasitriangular Lie bialgebras.
A Lie bialgebra $({\frak g},{\frak g}^{*})$ is called a {\em factorizable } if the following conditions are satisfied (see
\cite{Dm,fact}):
\begin{enumerate}
\item
${\frak g}${\em \ is equipped with a non--degenerate invariant
scalar product} $\left\langle \cdot ,\cdot \right\rangle$.

We shall always identify ${\frak g}^{*}$ and ${\frak g}$ by means
of this scalar product.

\item  {\em The dual Lie bracket on }${\frak g}^{*}\simeq {\frak g}${\em \
is given by}
\begin{equation}
\left[ X,Y\right] _{*}=\frac 12\left( \left[ rX,Y\right] +\left[
X,rY\right] \right) ,X,Y\in {\frak g},  \label{rbr}
\end{equation}
{\em where }$r\in {\rm End}\ {\frak g}${\em \ is a skew symmetric,
with respect to the  non--degenerate invariant scalar product,
linear operator.}

\item  $r${\em \ satisfies} {\em the} {\em modified classical Yang-Baxter
equation:}
\begin{equation}
\left[ rX,rY\right] -r\left( \left[ rX,Y\right] +\left[
X,rY\right] \right) =-\left[ X,Y\right] ,\;X,Y\in {\frak g}{\bf .}
\label{cybe}
\end{equation}
\end{enumerate}

The skew--symmetric operator $r$ on $\g$ satisfying the modified classical Yang--Baxter equation is called the classical r--matrix (or simply r--matrix).

Define operators $r_\pm \in {\rm End}\ {\frak g}$ by
\[
r_{\pm }=\frac 12\left( r\pm id\right) .
\]
We shall need some properties of the operators $r_{\pm }$. Denote
by ${\frak b}_\pm$ and ${\frak n}_\mp$ the image and the kernel of
the operator $r_\pm $:
\begin{equation}\label{bnpm}
{\frak b}_\pm = Im~r_\pm,~~{\frak n}_\mp = Ker~r_\pm.
\end{equation}

The classical Yang--Baxter equation implies that $r_{\pm }$ ,
regarded as a mapping from ${\frak g}^{*}$ into ${\frak g}$, is a
Lie algebra homomorphism. Moreover, $r_{+}^{*}=-r_{-},$\ and
$r_{+}-r_{-}=id.$

If the tangent Lie bialgebra of a Poisson--Lie group $G$ is factorizable then one can describe the dual group $G^*$ in terms of $G$ as follows.
Put ${\frak {d}}={\frak g + {\g}}$ (direct sum of two copies).
The mapping
\begin{eqnarray}\label{imbd}
{\frak {g}}^{*}\rightarrow {\frak {d}}~~~:X\mapsto
(X_{+},~X_{-}),~~~X_{\pm }~=~r_{\pm }X
\end{eqnarray}
is a Lie algebra embedding. Thus we may identify ${\frak g^{*}}$
with a Lie subalgebra in ${\frak {d}}$.

Naturally, embedding (\ref{imbd}) extends to a homomorphism
$$
G^*\rightarrow G\times G,~~L\mapsto (L_+,L_-).
$$
We shall identify $G^*$ with the corresponding subgroup in
$G\times G$.

Now we explicitly describe Poisson structures on the Poisson--Lie
group $G$ and on its dual group $G^*$.

For every group $A$ with Lie algebra $\frak a$ and any function
$\varphi \in C^\infty (A)$ we define left and right gradients $\nabla
\varphi , \nabla^{\prime} \varphi$, which are $C^\infty$-functions on $A$ with values in  ${\frak a}^{*}$, by the formulae
\begin{eqnarray}
\xi ( \nabla \varphi (x))\ =\left( \frac d{dt}\right) _{t=0}\varphi (e^{t\xi }x),  \nonumber \\
\xi ( \nabla^{\prime} \varphi (x))=\left( \frac d{dt}\right)
_{t=0}\varphi (xe^{t\xi }),~~\xi \in {\frak {a}.}  \label{grad}
\end{eqnarray}
The canonical Poisson bracket on Poisson--Lie group $G$ with
factorizable tangent bialgebra $(\g, \g^*)$ has the form:
\begin{equation}
\{ \varphi ,\psi \}~~=~~\frac 12 \left\langle r \nabla
\varphi,\nabla \psi \right\rangle -~\frac 12\left\langle
r\nabla^{\prime} \varphi , \nabla^{\prime} \psi \right\rangle,
\label{pbr}
\end{equation}
where $r$ is the corresponding r--matrix.

The canonical Poisson bracket on the dual Poisson--Lie group $G^*$
can be described in terms of the original group $G$ and the
classical r--matrix $r$.
\begin{proposition}
Denote by $G_*$ the group $G$ equipped with the following Poisson
bracket
\begin{equation}
\left\{ \varphi ,\psi \right\} _* =\left\langle r \nabla
\varphi,\nabla \psi \right\rangle +\left\langle r \nabla^{\prime
}\varphi,\nabla^{\prime }\psi\right\rangle -2\left\langle r_{+}
\nabla^{\prime }\varphi,\nabla \psi\right\rangle -2\left\langle
r_{-} \nabla\varphi,\nabla^{\prime }\psi\right\rangle ,
\label{tau}
\end{equation}
where all the gradients are taken with respect to the original
group structure on $G$.

Then the map $q:G^* \rightarrow G_*$ defined by
\begin{equation}\label{q*}
q(L_+,L_-)=L_+L_-^{-1}
\end{equation}
is a Poisson mapping and the image of $q$ is a dense open subset
in $G_*$.
\end{proposition}

Now we recall some facts on Poisson reduction and Poisson group actions.
A Poisson group action of a Poisson--Lie group $A$ on
a Poisson manifold $M$ is a group action $A\times M\rightarrow M$
which is also a Poisson map (as usual, we suppose that $A\times M$
is equipped with the product Poisson structure). In \cite{RIMS} it
is proved that if the space $M/A$ is a smooth manifold, there
exists a unique Poisson structure on $M/A$ such that the canonical
projection $M\rightarrow M/A$ is a Poisson map.

The main example of Poisson group actions is the so--called
dressing action. The dressing action can be described as follows
(see \cite{Lu,RIMS}).
\begin{proposition}\label{dressingact}
Let $G$ be a Poisson--Lie group with
factorizable tangent Lie bialgebra, $G^*$ the dual group. Then
there exists a unique left Poisson group action
$$
G\times G^*\rightarrow G^*,~~(g,(L_+,L_-))\mapsto g\circ
(L_+,L_-)
$$
such that if  $q:G^* \rightarrow G_*$ is the map defined by formula
(\ref{q*}) then
$$
q(g\circ (L_+,L_-))=gL_-L_+^{-1}g^{-1},
$$
i.e. the conjugation map $G\times G_* \rightarrow G_*$ is a Poisson
group action of the Poisson--Lie group $G$ on the Poisson manifold
$G_*$.
\end{proposition}

The notion of Poisson group actions may be generalized as follows.
Let $A\times M \rightarrow M$ be a Poisson group action of a
Poisson--Lie group $A$ on a Poisson manifold $M$.
A subgroup $K\subset A$ is called { admissible} if the set $%
C^\infty \left( M\right) ^K$ of $K$-invariants is a Poisson subalgebra in $%
C^\infty \left( M\right)$. If the space $M/K$ is a smooth manifold, we
may identify the algebras $C^\infty(M/K)$ and $C^\infty \left(
M\right) ^K$.  Hence there exists a Poisson structure on $M/K$
such that the canonical projection $M\rightarrow M/K$ is a Poisson
map. The space $M/K$ is called the { reduced Poisson manifold.}

\begin{remark}\label{nsm}
In order to shorten the notation we shall say that $M/K$ inherits a Poisson structure from $M$ even in case when the quotient $M/K$ is not smooth. In this case the Poisson structure on $M/K$ should be understood in the sense that the set $C^\infty \left( M\right) ^K$ is a Poisson subalgebra in $C^\infty \left( M\right)$.
\end{remark}

The following proposition proved in \cite{RIMS} gives a sufficient criterion for $K$
to be an admissible subgroup of a Poisson-Lie group $A$.
\begin{proposition}
\label{admiss}Let $\left( {\frak a},{\frak a}^{*}\right) $ be the tangent
Lie bialgebra of $A$. A connected Lie subgroup $K\subset A$ with Lie algebra
${\frak k}\subset {\frak a}$ is admissible if ${\frak k}^{\perp }\subset
{\frak a}^{*}$ is a Lie subalgebra.
\end{proposition}
In particular, $A$ itself is admissible. (Note that $K\subset A$ is a
Poisson--Lie subgroup if and only if ${\frak k}^{\perp }\subset {\frak a}^{*}$ is
an { ideal}; in that case the tangent Lie bialgebra of $K$ is $\left(
{\frak k},{\frak a}^{*}/{\frak k}^{\bot }\right) .)$

Even if $M$ is symplectic the reduced Poisson bracket on $M/K$ is
usually degenerate. The difficult part of reduction is the
description of the symplectic leaves in $M/K$. In case of Hamiltonian
group actions the appropriate technique is provided by the use of the moment
map. Although a similar notion of the { nonabelian moment map} in
the context of Poisson group theory is also available \cite{Lu}, it
is less convenient. We shall use the most general Poisson reduction scheme suggested in \cite{Ma}.
This scheme generalizes various reduction procedures in Poisson and symplectic geometry.

Let $M$ be a Poisson manifold with Poisson bracket $\{\cdot,\cdot\}_M$, $C\subset M$ a submanifold and $i : C\rightarrow M$ the corresponding inclusion. Denote by $P_M\in \Gamma(\bigwedge^2T\,M)$(here and below $\Gamma$ stands for the space of sections of the corresponding bundle) the Poisson tensor associated to the Poisson bracket on $M$. Let
$E \subset TM|_C$ be a subbundle of the tangent bundle of $M$ restricted to $C$. Assume that

(A1) $E \bigcap TC$ is an integrable distribution in $TC$; $E \bigcap TC$ defines a foliation $\Phi$ on $C$.

(A2) The foliation $\Phi$ is regular, so the space of leaves $C/\Phi$ is a smooth manifold with projection $\pi: C \rightarrow C/\Phi$ being a submersion.

(A3) The bundle $E$ leaves the Poisson bracket of $M$ invariant in the sense that if $\varphi,\psi$ are smooth functions on $M$ with
differentials vanishing on $E$ then the differential of the function $\{\varphi,\psi\}_M$ also vanishes on $E$.

The triple $(M,C,E)$ is called Poisson reducible if $C/\Phi$ has a unique Poisson structure $\{\cdot ,\cdot \}_{C/\Phi}$ such
that for any (locally defined) smooth functions $\varphi,\psi$ on $C/\Phi$, and any (locally defined) smooth extensions
$\overline{\varphi},\overline{\psi}$ of $\varphi \circ \pi, \psi \circ \pi$ to $M$, with differentials vanishing on $E$, we have
\begin{equation}\label{Pbr}
\{\varphi ,\psi \}_{C/\Phi}\circ \pi=\{\overline{\varphi},\overline{\psi}\}_M \circ i = <P_M,d\overline{\varphi}\wedge d\overline{\psi}>\circ i.
\end{equation}

\begin{proposition}{\bf (\cite{Ma}, Theorem 2.2)}\label{Reduce}
Assume that conditions (A1)-(A3) are satisfied. Let $P_M^\sharp:T^*M \rightarrow TM$ be the map induced by the Poisson tensor $P_M$ of the Poisson manifold $M$, i.e. $\alpha(P_M^\sharp(\beta))=P_M(\alpha,\beta)$ for any $\alpha, \beta \in T^*M$. Then the triple $(M,C,E)$ is Poisson reducible if and only if
\begin{equation}\label{cond1}
P_M^\sharp(E^0)\subset TC+E,
\end{equation}
where for $x\in C$ $E^0_x=\{\alpha_x\in T^*_xM | \alpha_x|_{E_x}=0\}$, i.e. $E^0$ is the annihilator of $E$.
\end{proposition}

\begin{remark}
If the foliation $\Phi$ is not regular and the condition (A2) is not satisfied Proposition \ref{Reduce} still holds in the sense that one can define a unique Poisson bracket $\{\cdot ,\cdot \}_{C/\Phi}$ on the algebra of functions on $C$ annihilated by vector fields from $\Gamma(E \bigcap TC)$, and this bracket satisfies relation (\ref{Pbr}). In order to shorten the notation we shall say that $C/\Phi$ is equipped with a Poisson structure even in case when the foliation $\Phi$ is not regular and the space of leaves $C/\Phi$ is not a smooth manifold.
\end{remark}

\begin{remark}\label{nsmp}
Let $K\times M \rightarrow M$ be an action of a Lie group on a Poisson manifold $M$. Assume that the set $C^\infty \left( M\right) ^K$ of $K$-invariants is a Poisson subalgebra in $C^\infty \left( M\right)$. Let ${\frak H}$ be the distribution on $M$ spanned by the vector fields generated by the action of $K$ on $M$, and let $C$ be an integral manifold for the integrable distribution generated by the Hamiltonian vector fields of $K$-invariant smooth functions on $M$. Let $E$ be the restriction of ${\frak H}$ to $C$. Since the algebra $C^\infty \left( M\right) ^K$ is a Poisson subalgebra in $C^\infty \left( M\right)$ condition (A3) is satisfied for the triple $(M,C,E)$. Condition (\ref{cond1}) is also satisfied since Hamiltonian vector fields of $K$--invariant functions are tangent to $C$ by definition. Therefore if condition (A1) is fulfilled for the distribution $E \bigcap TC$ then by Proposition \ref{Reduce} the space of leaves $C/\Phi$ naturally acquires a Poisson structure which should be understood in the sense of the previous remark if the foliation $\Phi$ is not regular. $C/\Phi$ is a Poisson subspace of $M/K$, where in the nonsmooth case the Poisson structure on the quotients should be understood in the sense of the previous remark and of Remark \ref{nsm}.

If $C$ is invariant under the action of $K$ then the distribution $E \bigcap TC$ consists of tangent spaces to $K$--orbits, and condition (A1) is obviously satisfied. Note that in this case the space of leaves $C/\Phi$ is isomorphic to the quotient $C/K$.

The construction described above can be applied, in particular, in case when $K\times M \rightarrow M$ is the restriction of a Poisson group action $A\times M \rightarrow M$ of a Poisson--Lie group $A$ to an admissible subgroup $K\subset A$.
\end{remark}


\section{Poisson Reduction and deformed W algebras}

\setcounter{equation}{0}

In this section we equip the slices defined in Proposition \ref{prop2} with Poisson structures.
We keep the notation introduced in Sections \ref{slices1} and \ref{plg}.

Let $G$ be a complex simple algebraic group, $\g$ its Lie algebra, $\h\subset \g$ a Cartan subalgebra in $\g$.
Let $s\in W$ be an element of the Weyl group of the pair $(\g,\h)$, $\p$ and $\b$ the parabolic subalgebra and the Borel subalgebra associated to $s$ in Section  \ref{slices1}, $\n$ the nilradical of $\p$, $N\subset G$ the subgroup corresponding to $\n$.

Let $\k$ be the nilradical of ${\mathfrak b}$. Denote by $\opb$ and $\opk$ the opposite Borel and nilpotent subalgebras of $\g$. Let $\h_0^\perp$ be the orthogonal complement, with respect to the Killing form, in $\h$ to $\mathbb{C}\h_0=\{x\in \h, s(x)=x\}$.
Now consider the following operator $r$ on $\g$:
\begin{equation}\label{r}
r=P_{\k} - P_{\opk} +r_0,~r_0=\frac{1+s}{1-s}P_{\h_0^\perp},
\end{equation}
where $P_{\k},P_{\opk}$ and $P_{\h_0^\perp}$ are the orthogonal projection operators, with respect to the Killing form,
onto $\k, \opk$ and $\h_0^\perp$ in the direct vector space decomposition
\begin{equation}\label{decr}
\g=\k+\opk+\h_0^\perp+\mathbb{C}\h_0.
\end{equation}
By the classification theorem for classical r--matrixes proved in \cite{BD} the operator $r$
is a solution to the modified classical Yang-Baxter equation. Moreover, $r$ is skew--symmetric with respect to the Killing form. In case of r-matrix (\ref{r}) the subalgebras $\b_\pm$ and
${\frak n}_\pm$ introduced in (\ref{bnpm}) are $\b_+={\mathfrak b}$, $\b_-=\opb$,
${\frak n}_+=\k$, ${\frak n}_-=\opk$.

Now we equip the group $G$ with the standard Lie-Poisson bracket (\ref{pbr}) associated to r--matrix (\ref{r}). By Proposition
\ref{dressingact} the action by conjugations of the Poisson--Lie group $G$
on the Poisson manifold $G_*$ is Poisson if $G_*$ carries Poisson bracket (\ref{tau}) associated with the same r-matrix. We would like to show that the unipotent radical $N$ of the parabolic subgroup $P \subset G$ is an admissible subgroup of the Poisson--Lie group $G$, and hence the conjugation action of $G$ on $G_*$ can be restricted to the subgroup $N$ in such a way that the quotient $G_*/N$ carries a natural Poisson structure. Then we show that the slice $N_sZs^{-1}$ introduced in Proposition \ref{prop1} is a Poisson submanifold of the quotient $G_*/N$.

\begin{lemma}
\label{radmiss}
Let $G$ be a complex simple algebraic group, $\g$ its Lie algebra, $\h\subset \g$ a Cartan subalgebra in $\g$.
Let $s\in W$ be an element of the Weyl group of the pair $(\g,\h)$, $\p$ and $\b$ the parabolic subalgebra and the Borel subalgebra associated to $s$ in Section  \ref{slices1}, $\n$ the nilradical of $\p$, $N\subset G$ the subgroup corresponding to $\n$, $r$ the classical r-matrix (\ref{r}) on $\g$. Then $N\subset G$ is an admissible subgroup in the Poisson-Lie group $G$ equipped with the standard Poisson bracket (\ref{pbr}) associated to $r$.
\end{lemma}

\begin{proof} By Proposition \ref{admiss} it suffices to show that $\p={\n}^{\perp }\subset
{\frak g}^{*}$ is a Lie subalgebra. Recall that according to formula (\ref{imbd}) $\g^*$ can be identified with the Lie subalgebra of $\g + \g$ formed by the elements $(r_+X , r_-X)\in \g + \g, ~X\in \g$ . Using formula (\ref{r}) one can describe the Lie algebra $\g^*$ as follows
\begin{eqnarray}\label{g*}
\qquad  \g^*\simeq \{(X_++\frac{1}{2}(r_0+id)X_\h~, -X_-+\frac{1}{2}(r_0-id)X_\h)\in \g + \g, \\  X_+ \in {\frak k}, X_-\in\opk, X_{\h}\in
\h \}. \nonumber
\end{eqnarray}

Using description (\ref{g*}) of the Lie algebra $\g^*$  one can identify $\p={\n}^{\perp }\subset
{\frak g}^{*}$ with the following Lie subalgebra in $\g +\g$
\begin{eqnarray}\label{nperp}
\qquad \quad \p={\n}^{\perp }\simeq \{(Z_++\frac{1}{2}(r_0+id)Z_\h~, -Z_-+\frac{1}{2}(r_0-id)Z_\h), \\  Z_+ \in {\frak k}, Z_-\in {\opk}\cap \p, Z_{\h}\in
\h \}, \nonumber
\end{eqnarray}

Since $\k$ and ${\opk}\cap \p$ are Lie subalgebras in $\g$ and $\h$ normalizes both of them the linear subspace in $\g^*$ defined by the r.h.s. of formula (\ref{nperp}) is a Lie subalgebra of $\{(r_+X , r_-X)\in \g + \g, ~X\in \g \} \simeq \g^*$.

\end{proof}

Now we restrict the action of $G$ on $G_*$ by conjugations to the subgroup $N\subset G$. By
Lemma \ref{radmiss} and the remark before Proposition \ref{admiss}
the space $G_*/N$ inherits a reduced Poisson structure from $G_*$.

 Let $s\in G$ be the normal representative of the element $s\in W$. Let $Z$ and $N_s$ be the subgroups of $G$ defined by (\ref{defz}) and (\ref{defns}).
Then by Proposition \ref{prop2} the quotient $N_sZs^{-1}N/N\simeq N_sZs^{-1}$ is a subspace of the quotient $G_*/N$ which carries the reduced Poisson bracket. We shall show that actually $N_sZs^{-1}N/N\simeq N_sZs^{-1}\subset G_*/N$ is a Poisson submanifold in $G_*/N$.

\begin{theorem}\label{mainth}
Let $G$ be a complex simple algebraic group, $\g$ its Lie algebra, $\h\subset \g$ a Cartan subalgebra in $\g$.
Let $s\in W$ be an element of the Weyl group of the pair $(\g,\h)$, $\p$ and $\b$ the parabolic subalgebra and the Borel subalgebra associated to $s$ in Section  \ref{slices1}, $\n$ the nilradical of $\p$, $N\subset G$ the subgroup corresponding to $\n$, $r$ the classical r-matrix (\ref{r}) on $\g$. Let $Z$ and $N_s$ be the subgroups of $G$ defined by (\ref{defz}) and (\ref{defns}). Restrict the conjugation action of $G$ on $G_*$ to the subgroup $N$.
Then the Poisson structure on $G_*$ induces a reduced Poisson structure on $G_*/N$, and $N_sZs^{-1}$ is a Poisson submanifold of $G_*/N$.
\end{theorem}
\begin{proof}
The conjugation action of $G$ on $G_*$ is a Poisson group action.
By Lemma \ref{radmiss} $N\subset G$ is an admissible subgroup. Therefore the quotient $G_*/N$ is naturally equipped with a Poisson structure.

By Remark \ref{nsmp} in order to prove that $N_sZs^{-1}=N_sZs^{-1}N/N$ is a Poisson submanifold of $G_*/N$ it is sufficient to check that the Hamiltonian vector fields generated by
$N$-invariant functions on $G_*$ are
tangent to $N_sZs^{-1}N$.

Let $\varphi \in C^\infty \left( {{
G}_{*}}\right) ^{N}$. Then $\varphi \left( vg\right) =\varphi \left( gv\right) $ for all $v\in N$, $g\in G_*$, and hence $
Z(g)=\nabla \varphi(g) -\nabla ^{\prime }\varphi(g) \in \p$ for any $g\in G$. Since $
\nabla ^{\prime }\varphi (g)={\rm Ad}\;g^{-1} (\nabla \varphi (g))$
we can rewrite the Poisson bracket (\ref{tau}) on ${G}_{*}$ in the
following form:
\[
\left\{ \varphi ,\psi \right\}_*(g)=2\left\langle
r_+Z(g)-{\rm Ad}\;g  (r_-Z(g)),\nabla \psi (g)\right\rangle.
\]
Thus using the right trivialization of $TG$ the
Hamiltonian field generated by $\varphi$ can be written in the following form:
\begin{equation}\label{hamvf1}
\xi _\varphi \left( g\right) =r_+Z(g)-{\rm Ad}\;g(r_-Z(g)).
\end{equation}

According to formulas (\ref{imbd}) and (\ref{nperp}) we also have
\begin{equation}\label{rz}
r_+Z=Z_++\frac{1}{2}(r_0+id)Z_\h~,r_-Z= -Z_-+\frac{1}{2}(r_0-id)Z_\h,
\end{equation}
where $$Z_+ \in {\frak k}, Z_-\in {\opk}\cap \p, Z_{\h}\in \h$$
are the components of $Z$ with respect to the direct vector space decomposition
$$
\g=\k+\opk+\h.
$$

Now assume that $g\in N_sZs^{-1}N$, $g=n_szs^{-1}n$, $n_s\in N_s,n\in N$, $z\in Z$.
Then from (\ref{hamvf1}) and (\ref{rz}) we deduce
\begin{equation}\label{hamvf}
\xi _\varphi \left( n_scs^{-1}n \right) =Z_++r_+Z_\h-{\rm Ad}\;\left( n_szs^{-1}n\right)(-Z_-+r_-Z_\h).
\end{equation}

On the other hand, in
the right trivialization of $TG$ the tangent
space $T_{n_szs^{-1}n}NZs^{-1}N$ is identified with $\n+\z+{\rm Ad}\;\left( n_szs^{-1}n\right) \n=\n+\z+{\rm Ad}\;\left( n_scs^{-1}\right) \n$.

Now using the fact that $\z\subset \l$, where $\l$ is the Levi factor of $\p$, and $\n$ is an ideal in $\p$ one checks straightforwardly that the vector field (\ref{hamvf}) is contained in $T_{n_szs^{-1}n}N_sZs^{-1}N$ at each point $n_szs^{-1}n$, $n_s\in N_s,n\in N$, $z\in Z$. Indeed, observe that definitions of $\k$, $\opk$, $\p$ and $\z$  imply the following inclusions $\k\subset \n+\z$, ${\opk}\cap \p\subset \z$, and hence
\begin{equation}\label{inZ}
Z_+\in \n+\z\subset T_{n_szs^{-1}n}N_sZs^{-1}N, Z_-\in \z.
\end{equation}
Recalling that $\z\subset \l$ we deduce from the second inclusion in (\ref{inZ}) that ${\rm Ad}n(Z_-)-Z_-\in \n$. Therefore
\begin{equation}\label{z1}
{\rm Ad}\;\left( n_szs^{-1}\right)({\rm Ad}n(Z_-)-Z_-)\in {\rm Ad}\;\left( n_szs^{-1}\right)\n \subset T_{n_szs^{-1}n}N_sZs^{-1}N.
\end{equation}

From the second inclusion in (\ref{inZ}) we also have
\begin{equation}\label{z2}
{\rm Ad}\;\left( n_szs^{-1}\right)Z_-={\rm Ad}\;\left( n_sz\right)Z_-\in \n+\z\subset T_{n_szs^{-1}n}N_sZs^{-1}N.
\end{equation}
Combining (\ref{z1}) and (\ref{z2}) we conclude that
\begin{equation}\label{z3}
{\rm Ad}\;\left( n_szs^{-1}n\right)Z_-\in T_{n_szs^{-1}n}N_sZs^{-1}N.
\end{equation}

Now consider the remaining terms in the r.h.s of (\ref{hamvf}). Firstly, by the definition of $r$ we obviously have
\begin{equation}\label{rzo}
r_+(Z_\h)=\frac{1}{1-s}Z_{\h_0^\perp}+\frac{1}{2}Z_{\h_0}\in \h, r_-(Z_\h)=\frac{s}{1-s}Z_{\h_0^\perp}-\frac{1}{2}Z_{\h_0}\in \h,
\end{equation}
where $Z_{\h_0^\perp}$ and $Z_{\h_0}$ are the components of $Z_\h\in \h$ with respect to the orthogonal direct vector space decomposition $\h=\h_0^\perp+\mathbb{C}\h_0$. Similarly to (\ref{z1}) we deduce that
\begin{equation}\label{z4}
{\rm Ad}\;\left( n_szs^{-1}\right)({\rm Ad}n(r_-Z_\h)-r_-Z_\h)\in {\rm Ad}\;\left( n_szs^{-1}\right)\n \subset T_{n_szs^{-1}n}N_sZs^{-1}N.
\end{equation}
Since $r_-Z_\h\in \h$ and $\h$ normalizes $\z$ and $\n$ we also obtain
\begin{equation}\label{z5}
{\rm Ad}\;\left( n_szs^{-1}\right)(r_-Z_\h)-{\rm Ad}s^{-1}(r_-Z_\h)\in \n+\z \subset T_{n_szs^{-1}n}N_sZs^{-1}N.
\end{equation}
Formulas (\ref{rzo}) and inclusion $\mathbb{C}\h_0\subset \z$ imply that $r_+Z_\h-{\rm Ad}s^{-1}(r_-Z_\h)=Z_{\h_0}\in \z\subset T_{n_szs^{-1}n}N_sZs^{-1}N$, and hence recalling (\ref{z4}), (\ref{z5}) we conclude
\begin{equation}\label{z6}
r_+Z_\h-{\rm Ad}\;\left( n_szs^{-1}n\right)(r_-Z_\h)\in T_{n_szs^{-1}n}N_sZs^{-1}N.
\end{equation}
Finally combining the first inclusion in (\ref{inZ}) and inclusions (\ref{z3}), (\ref{z6}) we obtain
$$
\xi _\varphi \left( n_scs^{-1}n \right) =Z_++r_+Z_\h-{\rm Ad}\;\left( n_szs^{-1}n\right)(-Z_-+r_-Z_\h)\in T_{n_szs^{-1}n}N_sZs^{-1}N.
$$
This concludes the proof.

\end{proof}

Note that for any $X\in \g$ and any regular function $\varphi$ on $G$ the functions $\langle\nabla \varphi,X \rangle$ and  $\langle\nabla ^{\prime }\varphi,X \rangle$ are regular. Therefore the space of regular functions on $G$ is closed with respect to Poisson bracket (\ref{tau}).
Since by Proposition \ref{prop2} the projection $N_sZs^{-1}N\rightarrow N_sZs^{-1}$ induced by the map $G_*\rightarrow G_*/N$ is a morphism of varieties, Proposition \ref{Reduce} implies that the algebra of regular functions on $N_sZs^{-1}$ is closed under the reduced Poisson bracket defined on $N_sZs^{-1}$ in the previous theorem. We call this Poisson algebra the deformed Poisson W--algebra associated to the Weyl group element $s\in W$, or, more precisely, to the conjugacy class of $s\in W$, and denote it by $W_{s}(G)$.


\section{The algebra $W_{s}(G)$ in case of subregular slices and simple singularities}
\label{sl3}

\setcounter{equation}{0}

In this section we consider deformed Poisson W-algebras in case of slices of dimension $r+2$, where $r$ is the rank of the underlying simple algebraic group $G$.

Recall that the definition of transversal slices to adjoint orbits in a complex simple Lie algebra $\g$ and to conjugacy classes in a complex simple algebraic group
$G$ given in \cite{SL} was motivated by the study of simple singularities. Simple singularities appear in algebraic group theory as some singularities of the fibers
of the conjugation quotient map $\delta_G: G \rightarrow H/W$ generated by the inclusion $\mathbb{C}[H]^W\simeq
\mathbb{C}[G]^G\hookrightarrow \mathbb{C}[G]$, where $H$ is a maximal torus of $G$ and $W$ is the Weyl group of the pair $(G, H)$. Some fibers of this map are singular,
the singularities correspond to irregular elements of $G$, and one can study these singularities by restricting $\delta_G$ to certain transversal slices to conjugacy classes in $G$. Simple singularities can be identified with the help of the following proposition proved in \cite{SL}.

\begin{proposition}{\bf (\cite{SL}, Section 6.5)}\label{sld}
Let $S$ be a transversal slice for the conjugation action of $G$ on itself. Assume that $S$  has dimension $r+2$, where $r$ is the rank of $G$.
Then the fibers of the restriction of the adjoint quotient map to $S$, $\delta_G:S \rightarrow H/W$, are normal surfaces with isolated singularities.
A point $x\in S$ is an isolated singularity of such a fiber iff $x$ is subregular in $G$, and $S$ can be regarded as a deformation of this singularity.

Moreover, if $t\in H/W$, and $x\in S$ is a singular point of the fiber $\delta_G^{-1}(t)$, then $x$ is a rational double point of type $_h\Delta_i$ for a suitable
$i\in\{1,\ldots,m\}$, where $\Delta_i$ are the components in the decomposition of the Dynkin diagram $\Delta(t)$ of the centralizer $Z_G(t)$ of $t$ in $G$,
$\Delta(t)=\Delta_1\cup\ldots \cup \Delta_m$. If $\Delta_i$ is of type $A$, $D$ or $E$ then $_h\Delta_i=\Delta_i$; otherwise $_h\Delta_i$ is the homogeneous diagram
of type $A$, $D$ or $E$ associated to $\Delta_i$ by the rule $_hB_n=A_{2n-1}$, $_hC_n=D_{n+1}$, $_hF_4=E_6$, $_hG_2=D_4$.
\end{proposition}

Note that all simple singularities of types $A$, $D$ and $E$ were explicitly constructed in \cite{SL} using certain special transversal slices in complex simple Lie algebras $\g$ and the adjoint quotient map $\delta_{\g}: \g \rightarrow \h/W$ generated by the inclusion $\mathbb{C}[\h]^W\simeq
\mathbb{C}[\g]^G\hookrightarrow \mathbb{C}[\g]$, where $\h$ is a Cartan subalgebra of $\g$ and $W$ is the Weyl group of the pair $(\g, \h)$. Below following \cite{S1} we recall an alternative description of simple singularities in terms of transversal slices in algebraic groups. In our construction we shall use the transversal slices defined in Section \ref{slices1}. The corresponding elements $s\in W$ will be associated to subregular nilpotent elements $e$ in $\g$ via the Kazhdan--Lusztig map introduced in \cite{KL}(recall that $e\in \g$ is subregular if the dimension of its centralizer in $\g$ is equal to ${\rm rank}~\g+2$). The Kazhdan--Lusztig map is a certain mapping from the set of nilpotent adjoint orbits in $\g$ to the set of conjugacy classes in $W$. In this paper we do not need the definition of this map in the general case.
We shall only  describe  the values of this map on the subregular nilpotent adjoint orbits.

\begin{lemma}\label{subreg}{\bf (\cite{S1}, Lemma 4)}
Let $\g$ be a complex simple Lie algebra, $\b$ a Borel subalgebra of $\g$ containing a Cartan subalgebra $\h\subset \b$. Let $W$ be the Weyl group of the pair $(\g, \h)$. Denote by $\Gamma=\{\alpha_1,\ldots,\alpha_r\}$, $r={\rm rank}~\g$ the corresponding system of simple positive roots of $\g$ and by $\Delta$ the root system of $\g$. Fix a system of root vectors $e_{\alpha}\in \g$, $\alpha \in \Delta$. One can choose a representative $e$ in the unique subregular nilpotent adjoint orbit of $\g$ and a representative $s_e$ in the conjugacy class in $W$, which corresponds to $e$ under the Kazhdan--Lusztig map, as follows (below we use the convention of \cite{E} for the numbering of simple roots; for the exceptional lie algebras we give the type of $e$ according to classification \cite{BK} and the type of $s_e$ according to classification \cite{C}):

\begin{itemize}

\item
$A_{r}$, $\g=\mathfrak{sl}_{r+1}$,
$$e=e_{\alpha_1}+\ldots +e_{\alpha_{r-1}},$$
the class of $s_e$ is the Coxeter class in a root subsystem $A_{r-1}\subset A_{r}$, and all such subsystems are $W$--conjugate,
$$
s_e=s_{\alpha_1}\ldots s_{\alpha_{r-1}};$$

\item
$B_r$, $\g=\mathfrak{so}_{2r+1}$,
$$e=e_{\alpha_1}+\ldots +e_{\alpha_{r-2}}+e_{\alpha_{r-1}+\alpha_{r}}+e_{\alpha_{r}},$$
the class of $s_e$ is the Coxeter class in a root subsystem $D_{r}\subset B_{r}$, and all such subsystems are $W$--conjugate,
$$
s_e=s_{\alpha_1}\ldots s_{\alpha_{r-2}}s_{\alpha_{r-1}}s_{\alpha_{r-1}+2\alpha_{r}};$$

\item
$C_r$, $\g=\mathfrak{sp}_{2r}$,
$$e=e_{\alpha_1}+\ldots +e_{\alpha_{r-2}}+e_{2\alpha_{r-1}+\alpha_{r}}+e_{\alpha_{r}},$$
the class of $s_e$ is the Coxeter class in a root subsystem $C_{r-1}+A_1\subset C_{r}$, and all such subsystems are $W$--conjugate,
$$
s_e=s_{\alpha_1}\ldots s_{\alpha_{r-2}}s_{2\alpha_{r-1}+\alpha_{r}}s_{\alpha_{r}};$$

\item
$D_r$, $\g=\mathfrak{so}_{2r}$,
$$e=e_{\alpha_1}+\ldots +e_{\alpha_{r-4}}+e_{\alpha_{r-3}+\alpha_{r-2}}+e_{\alpha_{r-2}+\alpha_{r-1}}+e_{\alpha_{r-1}}+e_{\alpha_{r}},$$
the class of $s_e$ is the class $D(a_1)$,
$$
s_e=s_{\alpha_1}\ldots s_{\alpha_{r-2}}s_{\alpha_{r-1}}s_{\alpha_{r-2}+\alpha_{r-1}+\alpha_{r}};$$

\item
$E_6$, the type of $e$ is $E_6(a_1)$,
$$e=e_{\alpha_1} +e_{\alpha_{2}+\alpha_{3}}+e_{\alpha_{4}}+e_{\alpha_{5}}+e_{\alpha_{3}+\alpha_{6}}+e_{\alpha_{6}},$$
$s_e$ has type $E_6(a_1)$,
$$
s_e=s_{\alpha_1}s_{\alpha_{2}+\alpha_{3}}s_{\alpha_{4}}s_{\alpha_{5}}s_{\alpha_{3}+\alpha_{6}}s_{\alpha_{6}};$$

\item
$E_7$, the type of $e$ is $E_7(a_1)$,
$$e=e_{\alpha_1} + e_{\alpha_2} +e_{\alpha_{3}+\alpha_{4}}+e_{\alpha_{5}}+e_{\alpha_{6}}+e_{\alpha_{7}}+e_{\alpha_{4}+\alpha_{7}},$$
$s_e$ has type $E_7(a_1)$,
$$
s_e=s_{\alpha_1}s_{\alpha_2} s_{\alpha_{3}+\alpha_{4}}s_{\alpha_{5}}s_{\alpha_{6}}s_{\alpha_{7}}s_{\alpha_{4}+\alpha_{7}};$$

\item
$E_8$, the type of $e$ is $E_8(a_1)$,
$$e=e_{\alpha_1} + e_{\alpha_2} + e_{\alpha_3} +e_{\alpha_{4}+\alpha_{5}}+e_{\alpha_{5}+\alpha_{8}}+e_{\alpha_{6}}+e_{\alpha_{7}}+e_{\alpha_{8}},$$
$s_e$ has type $E_8(a_1)$,
$$
s_e=s_{\alpha_1}s_{\alpha_2}s_{\alpha_3} s_{\alpha_{4}+\alpha_{5}}s_{\alpha_{5}+\alpha_{8}}s_{\alpha_{6}}s_{\alpha_{7}}s_{\alpha_{8}};$$

\item
$F_4$, the type of $e$ is $F_4(a_1)$,
$$e=e_{\alpha_1} + e_{\alpha_2}  +e_{\alpha_{2}+2\alpha_{3}}+e_{\alpha_{3}+\alpha_{4}},$$
$s_e$ has type $B_4$,
$$
s_e=s_{\alpha_1}s_{\alpha_2}s_{\alpha_{2}+2\alpha_{3}}s_{\alpha_{3}+\alpha_{4}};$$

\item
$G_2$, the type of $e$ is $G_2(a_1)$,
$$e=e_{2\alpha_1+\alpha_2} + e_{\alpha_2},$$
$s_e$ has type $A_2$,
$$
s_e=s_{3\alpha_{1}+\alpha_{2}}s_{\alpha_{2}}.$$

\end{itemize}

\end{lemma}

Now using elements $s_e$ defined in Lemma \ref{subreg} and Proposition \ref{prop1} we construct transversal slices to conjugacy classes in the algebraic group $G$.
We start by fixing appropriate elements $h_i\in \h_i$, where $\h_i$ are the two--dimensional $s_e$--invariant subspaces arising in decomposition (\ref{hdec}), and $s_e$ acts on $\h_i$ as rotation with angle $\theta_i=\frac{2\pi}{m_i}\leq\pi$, $m_i\in \mathbb{N}$, or as the reflection with respect to the origin (which also can be regarded as rotation with angle $\pi$). We consider all cases listed in Lemma \ref{subreg} separately.

\begin{enumerate}

\item
$A_{r}$

Let $\h_{min}$ be the unique two--dimensional plane on which $s_e$ acts as rotation with the minimal among of all $\theta_i,i=1,\ldots,K$ angle $\theta_{min}=\frac{2\pi}{r}$. Order the terms in (\ref{hdec}) in such a way that $\h_{min}$ is labeled by the maximal index. One checks straightforwardly that $\overline{\Delta}_{min}=\Delta$ and that for any choice $h_{min}\in \h_{min}$ as in Section \ref{slices1} the length $l(s_e)$ of $s_e$ with respect to the corresponding system of simple positive roots is $r+1$, $l(s_e)=r+1$. Note that in this case $\h_0$ is one dimensional, $\h_0=\mathbb{R}\omega_r$, where $\omega_r$ is the fundamental weight corresponding to $\alpha_r$.

\item
$B_r$

Let $\h_{min}$ be the unique two--dimensional plane on which $s_e$ acts as rotation with the minimal among of all $\theta_i,i=1,\ldots,K$ angle $\theta_{min}=\frac{\pi}{r-1}$, and $\h_1=\mathbb{R}\alpha_r^\vee$. The Weyl group element $s_e$ acts on $\h_1$ as reflection with respect to the origin. Order the terms in (\ref{hdec}) in such a way that $\h_{min}$ is labeled by the maximal index. One checks straightforwardly that $\overline{\Delta}_{min}=\Delta\setminus \{\alpha_r,-\alpha_r\}$, $\overline{\Delta}_1=\{\alpha_r,-\alpha_r\}$, $\Delta=\overline{\Delta}_{min}\bigcup \overline{\Delta}_1$ and that for any choice $h_{min}\in \h_{min}$ and $h_1\in \h_1$ as in Section \ref{slices1} the length $l(s_e)$ of $s_e$ with respect to the corresponding system of simple positive roots is $r+2$, $l(s_e)=r+2$. Note that in this case $\h_0=0$.

\item
$C_r$

Let $\h_{min}$ be the unique two--dimensional plane on which $s_e$ acts as rotation with the minimal among of all $\theta_i,i=1,\ldots,K$ angle $\theta_{min}=\frac{\pi}{r-1}$, and $\h_1=\mathbb{R}\alpha_r^\vee$. The Weyl group element $s_e$ acts on $\h_1$ as reflection with respect to the origin. Order the terms in (\ref{hdec}) in such a way that $\h_{min}$ is labeled by the maximal index. One checks straightforwardly that $\overline{\Delta}_{min}=\Delta\setminus \{\alpha_r,-\alpha_r\}$, $\overline{\Delta}_1=\{\alpha_r,-\alpha_r\}$, $\Delta=\overline{\Delta}_{min}\bigcup \overline{\Delta}_1$ and that for any choice $h_{min}\in \h_{min}$ and $h_1\in \h_1$ as in Section \ref{slices1} the length $l(s_e)$ of $s_e$ with respect to the corresponding system of simple positive roots is $r+2$, $l(s_e)=r+2$. Note that in this case $\h_0=0$.

\item
$D_r$

A root subsystem $D_{r-2}+D_2\subset D_r$ is invariant under the action of $s_e$, and $s_e$ acts on $D_{r-2}$ $(D_2)$ as a Coxeter element of $D_{r-1}\supset D_{r-2}$ $(D_3\supset D_2)$, respectively, with respect to the natural inclusions.
Let $\h_{min}$ be the unique two--dimensional plane in the Cartan subalgebra corresponding to the root subsystem $D_{r-2}$ on which $s_e$ acts as rotation with the angle $\theta_{min}=\frac{\pi}{r-2}$, and let $\h_{1}$ be the unique two--dimensional plane in the Cartan subalgebra corresponding to the root subsystem $D_{2}$ on which $s_e$ acts as rotation with the angle $\theta_{1}=\frac{\pi}{2}$. Order the terms in (\ref{hdec}) in such a way that $\h_{min}$ is labeled by the maximal index. One checks straightforwardly that $\overline{\Delta}_{min}=\Delta\setminus D_2$, $\overline{\Delta}_1=D_2$, $\Delta=\overline{\Delta}_{min}\bigcup \overline{\Delta}_1$ and that for any choice $h_{min}\in \h_{min}$ and $h_1\in \h_1$ as in Section \ref{slices1} the length $l(s_e)$ of $s_e$ with respect to the corresponding system of simple positive roots is $r+2$, $l(s_e)=r+2$. Note that in this case $\h_0=0$.

\item
$E_6$, $E_7$, $E_8$, $G_2$

Let $\h_{min}$ be the unique two--dimensional plane on which $s_e$ acts as rotation with the minimal among of all $\theta_i,i=1,\ldots,K$ angle $\theta_{min}$, $\theta_{min}=\frac{2\pi}{9}$ for $E_6$, $\theta_{min}=\frac{\pi}{7}$ for $E_7$, $\theta_{min}=\frac{\pi}{12}$ for $E_8$, $\theta_{min}=\frac{\pi}{3}$ for $G_2$. Order the terms in (\ref{hdec}) in such a way that $\h_{min}$ is labeled by the maximal index. One checks straightforwardly that $\overline{\Delta}_{min}=\Delta$ and that for any choice $h_{min}\in \h_{min}$ as in Section \ref{slices1} the length $l(s_e)$ of $s_e$ with respect to the corresponding system of simple positive roots is $r+2$, $l(s_e)=r+2$. Note that in this case $\h_0=0$.

\item
$F_4$

The multiplicity of the eigenvalue $e^{\frac{\pi i}{3}}$ of $s_e$ with the minimal possible argument is $2$. Therefore there are infinitely many two--dimensional planes in $\h_{\mathbb{R}}$ on which $s_e$ acts as rotation with angle $\theta_{min}=\frac{\pi}{3}$. It turns out that one can choose a two--dimensional plane $\h_{min}\subset \h_{\mathbb{R}}$ on which $s_e$ acts as rotation with angle $\theta_{min}=\frac{\pi}{3}$ and which is not orthogonal to any root. Order the terms in (\ref{hdec}) in such a way that $\h_{min}$ is labeled by the maximal index. Then $\overline{\Delta}_{min}=\Delta$ and for any choice $h_{min}\in \h_{min}$ and $h_1\in \h_1$ as in Section \ref{slices1} the length $l(s_e)$ of $s_e$ with respect to the corresponding system of simple positive roots is $r+2$, $l(s_e)=r+2$. Note that in this case $\h_0=0$.

\end{enumerate}

\begin{remark}\label{br}
Note that in all cases 1-6 considered above the parabolic subalgebra $\p$ associated to $s_e$ with the help of the corresponding element $\bar{h}_0$ is a Borel subalgebra.
\end{remark}

The following proposition follows straightforwardly from Propositions \ref{prop2} and \ref{prop1}.

\begin{proposition}\label{subreg11}{\bf (\cite{S1}, Proposition 5)}
Let $G$ be a complex simple algebraic group with Lie algebra $\g$, $H$ is a maximal torus of $G$, $\h$ the Cartan subalgebra in $\g$ corresponding to $H$.
Let $e$ be the subregular nilpotent element in $\g$ defined in the previous lemma and $s_e$ the element of the Weyl group of the pair $(\g,\h)$ associated to $e$ in Lemma \ref{subreg}. Denote by $s\in G$ the normal representative of $s_e$ in $G$ constructed in Section \ref{slices1}. Denote by $\b$ the Borel subalgebra of $\g$ associated to $s_e$ in Remark \ref{br}.
Let $B$ be the Borel subgroup of $G$ corresponding to the Borel subalgebra $\b$, $N$ the unipotent radical of $B$, $\overline B$ the opposite Borel subgroup in $G$, $\overline{N}$ the unipotent radical of $\overline B$.

Then the variety $N_sZs^{-1}$, where $Z=\{z\in H\mid szs^{-1}=z\}$, $N_s=\{n\in N\mid sns^{-1}\in \overline N\}$, is a transversal slice  to the set of conjugacy classes in $G$ and the conjugation map $N\times N_sZs^{-1}\rightarrow NZs^{-1}N$ is an isomorphism of varieties. The slice $N_sZs^{-1}$ has dimension $r+2$, $r=~{\rm rank}~\g$.
\end{proposition}

An immediate corollary of the last proposition and of Proposition \ref{sld} is the following construction of simple singularities in terms of the conjugation quotient map.

\begin{proposition}\label{subreg22}{\bf (\cite{S1}, Proposition 6)}
Let $G$ be a complex simple algebraic group with Lie algebra $\g$, $H$ is a maximal torus of $G$, $\h$ the Cartan subalgebra in $\g$ corresponding to $H$.
Let $e$ be the subregular nilpotent element in $\g$ defined in the Lemma \ref{subreg} and $s_e$ the element of the Weyl group of the pair $(\g,\h)$ associated to $e$ in Lemma \ref{subreg}. Denote by $s\in G$ the representative of $s_e$ in $G$ constructed in Section \ref{slices1}. Let $\b$ be the Borel subalgebra of $\g$ associated to $s_e$ in Remark \ref{br}.
 Denote by $B$ the Borel subgroup of $G$ corresponding to the Borel subalgebra $\b$, and by $N$  the unipotent radical of $B$. Let $\overline B$ be the opposite Borel subgroup in $G$, $\overline{N}$ the unipotent radical of $\overline B$
 $Z=\{z\in H\mid szs^{-1}=z\}$ and $N_s=\{n\in N\mid sns^{-1}\in \overline N\}$.

Then fibers of the restriction of $\delta_G:G\rightarrow G/H$ to $N_sZs^{-1}$ are normal surfaces with isolated singularities. A point $x\in N_sZs^{-1}$ is an isolated singularity of such a fiber iff $x$
is subregular in $G$.

Moreover, if  $t\in H/W$, and $x\in N_sZs^{-1}$ is a singular point of the fiber $\delta_G^{-1}(t)$, then $x$ is a rational double point of type $_h\Delta_i$ for a suitable $i\in\{1,\ldots,m\}$, where $\Delta_i$ are the components in the decomposition of the Dynkin diagram $\Delta(t)$ of the centralizer $Z_G(t)$ of $t$ in $G$, $\Delta(t)=\Delta_1\cup\ldots \cup \Delta_m$. If $\Delta_i$ is of type $A$, $D$ or $E$ then $_h\Delta_i=\Delta_i$; otherwise $_h\Delta_i$ is the homogeneous diagram of type $A$, $D$ or $E$ associated to $\Delta_i$ by the rule $_hB_n=A_{2n-1}$, $_hC_n=D_{n+1}$, $_hF_4=E_6$, $_hG_2=D_4$.
\end{proposition}

Theorem \ref{mainth} implies that any slice of the form $N_sZs^{-1}$ inherits a Poisson structure from the Poisson manifold $G_*$ associated to r--matrix (\ref{r}). On the other hand by Proposition \ref{subreg22} the slices $N_sZs^{-1}$ of dimensions $r+2$ introduced in Proposition \ref{subreg11} can be regarded as deformations of simple singularities. Therefore the corresponding algebras $W_{s}(G)$ can be regarded as Poisson deformations of simple singularities. Note that the regular functions which are invariant with respect to the action of $G$ on $G_*$ by conjugations lie in the center of the Poisson algebra of regular functions on $G_*$ (see \cite{dual}). The restrictions of these functions to the slices $N_sZs^{-1}$ lie in the center of the Poisson algebra  $W_{s}(G)$. Therefore the fibers of the conjugation quotient map $\delta_G:N_sZs^{-1}\rightarrow H/W$ also inherit Poisson structures from $N_sZs^{-1}$. In particular, the Poisson structures on the singular fibers of $\delta_G$ generate Poisson brackets on the coordinate rings of Kleinian singularities.

In general the Poisson structure on $N_sZs^{-1}$ and the Poisson bracket of the corresponding algebra $W_{s}(G)$ are difficult to calculate. We shall consider the simplest nontrivial example when one can do that explicitly.

The first nontrivial example of Poisson deformed W--algebras related to simple singularities appears in case when $\g={\mathfrak{sl}_3}$ and $s$ is the reflection with respect to a single root.
We are going to describe the corresponding slice $N_sZs^{-1}$ and the algebra $W_{s}({SL_3})$ explicitly.
We use the usual matrix realization of the Lie algebra ${\mathfrak{sl}_3}$ by complex $3\times 3$ traceless matrices,
$$
{\mathfrak{sl}_3}=\{X\in Mat_3(\mathbb{C}), {\rm tr}~X=0\},
$$
and take $\h$ to be the subalgebra of traceless diagonal matrices.

If the system of simple positive roots $\Gamma$ of ${\mathfrak{sl}_3}$ is fixed as in case 1 considered after Lemma \ref{subreg} then $s$ becomes the longest element in the Weyl group of the pair $(\g,\h)$ with respect to $\Gamma$, and $\n$ is the nilradical of $\g$ which can be realized as the subalgebra of lower triangular matrices.

The representative $s\in SL(3)$ of the element $s\in W$ that we introduced in Section \ref{slices1} looks as follows:
\begin{equation}\label{s}
s=\left(
    \begin{array}{ccc}
      0 & 0 & 1 \\
      0 & -1 & 0 \\
      1 & 0 & 0 \\
    \end{array}
  \right).
\end{equation}
The Lie group $Z$ of the $\h$--component $\z$ of the centralizer of $s$ consists of the matrices of the following form
\begin{equation}\label{Z}
Z=\{\left(
  \begin{array}{ccc}
    t & 0 & 0 \\
    0 & t^{-2} & 0 \\
    0 & 0 & t \\
  \end{array}
\right),~t\in \mathbb{C}^*\},
\end{equation}
and the Lie group $N_s$ coincides with $N$ in this case, $N=N_s=\{n\in N\mid sns^{-1}\in \overline N\}$.
The transversal slice $N_sZs^{-1}$ can be described explicitly using (\ref{s}) and (\ref{Z}),
\begin{equation}
N_sZs^{-1}=\{S=\left(
  \begin{array}{ccc}
    0 & 0 & t \\
    0 & -t^{-2} & \alpha \\
    t & \beta & \gamma \\
  \end{array}
\right),~t\in \mathbb{C}^*,~\alpha,\beta,\gamma\in \mathbb{C}\}.
\end{equation}

The algebra of regular functions on the slice $N_sZs^{-1}$ is generated by four functions $\varphi_\alpha,\varphi_\beta,\varphi_\gamma,\varphi_t$,
$$
\varphi_\alpha(S)=\alpha,~\varphi_\beta(S)=\beta,~\varphi_\gamma(S)=\gamma,~\varphi_t(S)=t,~S\in N_sZs^{-1}.
$$
In order to describe the Poisson algebra $W_{s}({SL_3})$ we have to calculate the Poisson brackets of these functions. According to Proposition \ref{Reduce} and formula \ref{Pbr} we have to find $N$-invariant extensions of $\varphi_\alpha, \varphi_\beta, \varphi_\gamma, \varphi_t$ to $N_sZs^{-1}N$ and then calculate their Poisson brackets in the Poisson algebra of regular functions on $G_*$.

Introduce functions $\overline{\varphi}_\alpha, \overline{\varphi}_\beta, \overline{\varphi}_\gamma, \overline{\varphi}_t$ on $G_*$ as follows
\begin{eqnarray}\label{ext}
  \overline{\varphi}_\alpha (g)&=&g_{13}^2(g_{11}g_{23}-g_{13}g_{21})+g_{23},  \nonumber \\
 \overline{\varphi}_\beta (g)&=& g_{32}-g_{12}g_{13}^{-3}(1+g_{13}^2g_{33}), \\
  \overline{\varphi}_\gamma (g) &=& g_{11}+g_{22}+g_{33}+g_{13}^{-2}, \nonumber \\
  \overline{\varphi}_t (g) &=& g_{13}, \nonumber
\end{eqnarray}
where
\begin{equation}\label{g}
g=\left(
    \begin{array}{ccc}
      g_{11} & g_{12} & g_{13} \\
      g_{21} & g_{22} & g_{23} \\
      g_{31} & g_{32} & g_{33} \\
    \end{array}
  \right)\in SL(3).
\end{equation}

One can check straightforwardly that the restrictions of the functions $\overline{\varphi}_\alpha,\overline{\varphi}_\beta,\overline{\varphi}_\gamma,\overline{\varphi}_t$ to $N_sZs^{-1}N$ are invariant with respect to the action of $N$ on $N_sZs^{-1}N$ by conjugations and that their restrictions to $N_sZs^{-1}$ coincide with $\varphi_\alpha,\varphi_\beta,\varphi_\gamma,\varphi_t$, respectively.

Now according to Proposition \ref{Reduce} and formula (\ref{Pbr}) the Poisson brackets of the functions $\varphi_\alpha,\varphi_\beta,\varphi_\gamma,\varphi_t$ in the Poisson algebra $W_{s}({SL_3})$ are equal to the restrictions to $N_sZs^{-1}$ of the Poisson brackets of $\overline{\varphi}_\alpha,\overline{\varphi}_\beta,\overline{\varphi}_\gamma,\overline{\varphi}_t$ regarded as functions on the Poisson manifold $G_*$. These Poisson brackets can be calculated using Poisson brackets of functions $\varphi_{ij}$, $\varphi_{ij}(g)=g_{ij}$, where $g\in SL(3)$ is given by formula (\ref{g}).

From formula (\ref{tau}) and from the definition (\ref{r}) of the r-matrix we have
\begin{eqnarray}\label{gst}
\{\varphi_{ij},\varphi_{km}\}=\varphi_{im} \varphi_{kj} (\varepsilon_{ik}+\varepsilon_{mj})+2\delta_{im}\sum_{l>i}\varphi_{kl}\varphi_{l j}- \\
-2\delta_{jk}\sum_{l>j}\varphi_{il}\varphi_{lm}+\varphi_{ij}\varphi_{km}(\delta_{im}-\delta_{jk}),\nonumber
\end{eqnarray}
$$
\varepsilon_{ij}=\left\{ \begin{array}{c}
                     1 ~{\rm if}~i>j \\
                     0 ~{\rm if}~i=j \\
                     -1 ~{\rm if}~i<j
                   \end{array}\right. .
$$

Now Proposition \ref{Reduce}, formulas (\ref{gst}) and (\ref{ext}) imply that
\begin{eqnarray}\label{wps}
 \{\varphi_\alpha, \varphi_\beta \}&=&2(\varphi_t^2-\varphi_t^{-4}-\varphi_t^{-2}\varphi_\gamma), \nonumber \\
 \{\varphi_t, \varphi_\alpha \}&= &\varphi_t \varphi_\alpha, \nonumber\\
  \{\varphi_t, \varphi_\beta \}&= &-\varphi_t \varphi_\beta, \\
   \{\varphi_t, \varphi_\gamma \}&=&0,\nonumber\\
   \{\varphi_\gamma, f \}&=&\{\varphi_t^{-2}, f \} \nonumber
\end{eqnarray}
for any regular function $f=f(\alpha,\beta,\gamma,t)$.
Poisson brackets (\ref{wps}) completely determine the Poisson structure of the Poisson algebra $W_{s}({SL_3})$.

Now we describe the associated Poisson structures on the fibers of the conjugation quotient map $\delta_G :N_sZs^{-1}\rightarrow H/W$, where $H$ is the maximal torus in $G$ with Lie algebra $\h$.
The fibers of the adjoint quotient map are the intersections of the level surfaces of the regular class functions on $G$ with the slice $N_sZs^{-1}$.

Recall that the regular functions which are invariant with respect to the action of $G$ on $G_*$ by conjugations lie in the center of the Poisson algebra of regular functions on $G_*$ (see \cite{dual}). The restrictions of these functions to the slice $N_sZs^{-1}$ lie in the center of the Poisson algebra  $W_{s}({SL_3})$.

In case of $\g={\mathfrak{sl}_3}$ there are two algebraically independent $G$--invariant functions on $G_*$, $\overline{\varphi}_1(g)={\rm tr}g$ and $\overline{\varphi}_2(g)={\rm tr}g^2$. The restrictions ${\varphi}_1$ and ${\varphi}_2$ of these functions to $N_sZs^{-1}$ can be expressed in terms of $\varphi_\alpha,\varphi_\beta,\varphi_\gamma,\varphi_t$ as follows
\begin{eqnarray}
{\varphi}_1&=&\varphi_\gamma- \varphi_t^{-2}, \\
 {\varphi}_2&=& 2\varphi_t^{2}+\varphi_t^{-4}+2\varphi_\alpha\varphi_\beta +\varphi_\gamma^2. \nonumber
\end{eqnarray}

If we denote by $Z_{c_1,c_2}$ the Poisson ideal in $W_{s}({SL_3})$ generated by the elements
${\varphi}_1-c_1$ and ${\varphi}_2-c_2$ then the quotient $W_{s}({SL_3})/Z_{c_1,c_2}$ is a Poisson algebra which can be regarded as Poisson algebra of functions on the fiber of the conjugation quotient map $\delta_G:N_sZs^{-1}\rightarrow H/W$ defined by the equations ${\varphi}_1(S)=c_1,~{\varphi}_2(S)=c_2,~S\in N_sZs^{-1}$.

In particular if $c_1=c_2=3$ the corresponding fiber is singular and lies in the unipotent variety of $G$. If we introduce new variables $x=t^2\alpha,~y=t^2\beta,~z=t^2+1$ then the elements of the singular fiber satisfy the following equations in $N_sZs^{-1}$
\begin{eqnarray}\label{ss}
\gamma = \frac{1}{z-1}+3, z\neq 1, \\
  z^3+ x y = 0. \nonumber
\end{eqnarray}
The first equation in (\ref{ss}) allows to eliminate $\gamma$, and the second equation defines $A_2$--type simple singularity according to the A-D-E classification.

If we introduce the functions $\varphi_x,\varphi_y,\varphi_z\in W_s({SL_3})/Z_{3,3}$ on the singular fiber of the conjugation quotient map by
$$
\varphi_x(x,y,z)=x,\varphi_y(x,y,z)=y,\varphi_z(x,y,z)=z
$$
then their Poisson brackets in the Poisson algebra $W_{s}({SL_3})/Z_{3,3}$ take the form
\begin{eqnarray}\label{brsing}
\{\varphi_x, \varphi_y \}&=&6(\varphi_z-1)\varphi_z^2, \nonumber \\
 \{\varphi_z, \varphi_x \}&= &2(\varphi_z-1)\varphi_x, \\
  \{\varphi_z, \varphi_y \}&= &-2(\varphi_z-1)\varphi_y. \nonumber\\
\end{eqnarray}

Poisson structure (\ref{brsing}) on the singular fiber of the conjugation quotient map is proportional to the Poisson structure on the singular fiber of the adjoint quotient map $\delta_{\g}:s(e)\rightarrow \h/W$ derived in \cite{Pr} for the Slodowy slice $s(e)$ at a subregular element $e\in {\mathfrak{sl}_3}$, the coefficient proportionality being $\varphi_z-1$.


\section{The algebra $W_{s}(G)$ in case when $s$ is the reflection with respect to a long root}
\label{long}

\setcounter{equation}{0}

In this section we explicitly describe the Poisson structure of the algebra $W_s(G)$ in case when $s$ is the reflection with respect to a long root in the root system of the Lie algebra $\g$ of a complex simple algebraic group $G$.

Let
$\h$ be a Cartan subalgebra of $\g$, and let $\Delta$ be the root
system of $\g$ relative to $\h$.  Let
$\Gamma'=\{\alpha_1,\ldots,\alpha_\ell\}$ be a basis of simple roots
in $\Delta$.  If $\g$ is not of type $A_r$ or $C_r$, there is a
unique long root in $\Gamma'$ linked with the lowest root on the extended Dynkin diagram of $\g$; we
call it $\beta$. For $\g$ of type $A_r$ and $C_r$ we set
$\beta=\alpha_r$. Let $s=s_\beta$ be the reflection with respect to the root $\beta$ in the Weyl group of the pair $(\g,\h)$. For $s=s_\beta$ decomposition (\ref{hdec}) contains two terms: $\h_1=\mathbb{R}\beta^\vee$, $\h_0=\h_1^\perp$, where $\h_1^\perp$ is the orthogonal complement to $\h_1$ in $\h$ with respect to the Killing form. A parabolic subalgebra $\p$ associated to $s$ with the help of the decomposition $\h_{\mathbb{R}}=\h_1+\h_0$ can be constructed as follows.

Choose root vectors $e,f \in
\g$ corresponding to roots $\beta$ and $-\beta$ such that
$(e,[e,f],f)$ is an ${\frak sl}_2$-triple
and put $h=[e,f]$. Obviously we have $h\in \h_1$. Put $h_1=h$ and choose any element $h_0\in \h_0$ such that $h_0(\alpha)\neq 0$ for any root $\alpha \in \Delta$ which is not orthogonal to the $s$--invariant subspace $\h_0$ with respect to the natural pairing between $\h_{\mathbb{R}}$ and $\h_{\mathbb{R}}^*$.
Assume also that conditions (\ref{cond}) are satisfied for $h_0$ and $h_1$. Then by the results of Section \ref{slices1} the element $\bar{h}=h_1+h_0$ belongs to a Weyl chamber. Let $\Delta_+$ and $\Gamma$  be the corresponding system of positive roots and of simple positive roots, respectively. Denote by $\b$ the corresponding Borel subalgebra.

The action of the inner derivation ${\rm ad}~ h$ gives $\g$ a short
$\mathbb{Z}$-grading
\begin{equation}\label{sgrad}
\g\,=\,(\g)_{-2}\oplus(\g)_{-1}\oplus(\g)_{0}\oplus(\g)_{1}\oplus(\g)_{2},\qquad\,
(\g)_{m}=\{x\in\g\,|\,\,[h,x]=mx\}
\end{equation}
with $(\g)_{1}\oplus(\g)_{2}$ and
$(\g)_{-1}\oplus (\g)_{-2}$ being Heisenberg Lie algebras. One knows that $(\g)_{\pm 2}$ is spanned by $e$ and $f$, respectively.
Denote by $\z_e$ the centralizer of $e$ in $\g$. The Lie algebra $\z_e$ inherits a
$\mathbb{Z}$-grading from $\g$, $\z_e=(\z_e)_0\oplus(\z_e)_1\oplus (\z_e)_2$,
$(\z_e)_m=(\g)_m$ for $m=1,2$, and the component
 $(\z_e)_0$  is the orthogonal
complement to $\mathbb{C} h$ in $(\g)_0$ with respect to the Killing form. In particular, $(\z_e)_0$ is an
ideal of codimension $1$ in the Levi subalgebra $(\g)_0$.

From the above considerations it also follows that $\beta$ is the longest root in $\Delta$ with respect to the system $\Gamma$ of simple positive roots.

Following the recipe of Section \ref{slices1} we define the parabolic subalgebra $\p$ with the help of the element $\bar{h}_0=h\in \h_1$, $\p=(\g)_{-2}\oplus(\g)_{-1}\oplus(\g)_{0}$. Let $\n=(\g)_{-2}\oplus(\g)_{-1}$ be the nilradical of $\p$ and $\l=(\g)_0$ the Levi factor of $\p$. Let $N$ be the subgroup of $G$ corresponding to the Lie subalgebra $\n\subset \g$.

Note that since  $(\z_e)_0$  is the orthogonal complement to $\mathbb{C} h$ in $(\g)_0$ the Lie subalgebra $\z\subset \l$ in the considered case can be identified with
$(\z_e)_0$, $\z=(\z_e)_0$. Moreover, from the definition of grading (\ref{sgrad}) it follows that the Lie algebra $\n_s$ coincides in this case with $\n$, and $N_sZ=Z_e$, where $Z_e$ is the centralizer of $e$ in $G$ with respect to the adjoint action.

From Propositions \ref{prop1} and \ref{prop2} we deduce that the variety $N_sZs^{-1}$ is a transversal slice to the set of conjugacy classes in $G$, and the conjugation map $N\times N_sZs^{-1}\rightarrow NZs^{-1}N$ is an isomorphism of varieties.

Now let ${\frak k}$ the nilradical of the Borel subalgebra $\b$, and ${\opk}$ the opposite nilpotent Lie subalgebra of $\g$. Since the subspace $\h_0^\perp =\mathbb{C}h$ is one dimensional and $s=-1$ on $\h_0^\perp$ r--matrix (\ref{r}) is reduced in the considered case to the standard r-matrix on $\g$, $r=P_{{\frak k}}-P_{{\opk}}$ (see \cite{Dm}). By Theorem \ref{mainth} the slice $N_sZs^{-1}=Z_es^{-1}$ inherits a Poisson structure from the Poisson manifold $G_*$ associated to r--matrix $r=P_{{\frak k}}-P_{{\opk}}$. Using Proposition \ref{Reduce} we shall explicitly calculate the Poisson structure of the corresponding algebra $W_s(G)$.

To each function $\varphi$ on the slice $Z_es^{-1}$ we associate its $N$--invariant extension $\varphi^*$ to $NZs^{-1}N$. Since for $N$ and $s$ fixed above the map inverse to (\ref{cross}) sends every element $nzs^{-1}n'$ to $({n'}^{-1},n'nzs^{-1})$, $n,n'\in N$, $z\in Z$, we have
\begin{equation}\label{extf}
\varphi^*(nzs^{-1}n')=\varphi(n'nzs^{-1})
\end{equation}
We shall also use the function $\varphi'$ on $Z_e$ associated to $\varphi$ by the formula $\varphi'(nz)=\varphi(nzs^{-1})$, $n\in N$, $z\in Z$.

Let $\varphi, \psi$ be two functions on the slice $Z_es^{-1}$. In order to calculate their Poisson bracket in the algebra $W_s(G)$ we have to find, according to Proposition \ref{Reduce}, the differentials $d\varphi^*, d\psi^*\in T^*(NZs^{-1}N)=T^*(G_*)/N_{NZs^{-1}N}$, $N_{NZs^{-1}N}=\{\alpha\in  T^*G_*: \alpha_x|_{T_x(NZs^{-1}N)}=0, x\in NZs^{-1}N\}$, and then substitute their arbitrary representatives $d\overline{\varphi},d\overline{\psi}$ in $T^*G_*$ into formula (\ref{Pbr}). We shall use the left trivialization of the tangent bundle $TG$. In this trivialization the tangent space $T_{nzs^{-1}}Z_es^{-1}$ to the slice $Z_es^{-1}$ at point $nzs^{-1}$ is identified with $\n+\z+{\rm Ad}(sz^{-1}n^{-1})\n=\n+\z+\opn$, $T_{nzs^{-1}}G=\n+\z+\opn$.

Let $e_i$ be a basis of $\n$, $e_i^*$ the dual basis of $\opn$, $f_i$ a basis of $\z$, and $f_i^*$ the dual basis of $\z$. Put $g=nzs^{-1}$. Recalling definition (\ref{extf}) and the fact that $s$ centralizes $\z$ we have
\begin{eqnarray}
 d\varphi^*(g)= \frac{d}{dt}|_{t=0}\big(\sum_i\varphi^*(ge^{te_i})e_i^*+\sum_i\varphi^*(ge^{tf_i})f_i^*+\sum_i\varphi^*(ge^{te_i^*})e_i\big)= \nonumber \\
 =\frac{d}{dt}|_{t=0}\big(\sum_i\varphi^*(e^{te_i}g)e_i^*+\sum_i\varphi^*(nze^{tf_i}s^{-1})f_i^*+
 \nonumber\\
 +\sum_i\varphi^*(nze^{t{\rm Ad}s^{-1}e_i^*}s^{-1})e_i\big)= \label{diff*} \\
 =\frac{d}{dt}|_{t=0}\big(\sum_i\varphi'(e^{te_i}nz)e_i^*+\sum_i\varphi'(nze^{tf_i})f_i^*+
 \sum_i\varphi'(nze^{t{\rm Ad}s^{-1}e_i^*})e_i\big)= \nonumber \\
 =P_{\opn}\nabla\varphi'(nz)+{\rm Ad}s\nabla'\varphi'(nz), \nonumber
\end{eqnarray}
where $P_{\opn}$ is the orthogonal projector onto the subspace $\opn$ in $\g$, and $\nabla\varphi'$, $\nabla'\varphi'$ are the left (right) gradients of $\varphi'$ regarded as a function on the Lie group $Z_e$.

Now using formula (\ref{diff*}), Proposition \ref{Reduce}, formula (\ref{tau}) for the Poisson bracket on $G_*$ and the relation $\nabla\varphi(g)={\rm Ad}g\nabla'\varphi(g)$ we obtain that
\begin{eqnarray}
\{\varphi,\psi\}(nzs^{-1})= \left\langle r \nabla
\varphi'(nz),\nabla \psi'(nz) \right\rangle +\left\langle r \nabla^{\prime
}\varphi'(nz),\nabla^{\prime }\psi'(nz)\right\rangle- \nonumber \\
-2\left\langle r_{-}
\nabla^{\prime }\varphi'(nz),\nabla \psi'(nz)\right\rangle -2\left\langle
r_{+} \nabla\varphi'(nz),\nabla^{\prime }\psi'(nz)\right\rangle- \label{str*} \\
-2\left\langle \nabla^{\prime
}\varphi'(nz),{\rm Ad}s\nabla\psi'(nz)\right\rangle+2\left\langle \nabla\varphi'(nz),{\rm Ad}s\nabla^{\prime}\psi'(nz)\right\rangle+\nonumber \\
+\left\langle {\rm Ad}(nz)\nabla^{\prime
}\varphi'(nz),\nabla\psi'(nz)\right\rangle-\left\langle {\nabla\varphi'(nz),\rm Ad}(nz)\nabla^{\prime
}\psi'(nz)\right\rangle. \nonumber
\end{eqnarray}

Identifying the slice $Z_es^{-1}$ with the Lie group $Z_e$ we can simply write the formula for the Poisson bracket on $Z_e$ induced by Poisson structure (\ref{str*}),
\begin{eqnarray*}
\{\varphi',\psi'\}(nz)= \left\langle r \nabla
\varphi'(nz),\nabla \psi'(nz) \right\rangle +\left\langle r \nabla^{\prime
}\varphi'(nz),\nabla^{\prime }\psi'(nz)\right\rangle- \nonumber \\
-2\left\langle r_{-}
\nabla^{\prime }\varphi'(nz),\nabla \psi'(nz)\right\rangle -2\left\langle
r_{+} \nabla\varphi'(nz),\nabla^{\prime }\psi'(nz)\right\rangle-  \\
-2\left\langle \nabla^{\prime
}\varphi'(nz),{\rm Ad}s\nabla\psi'(nz)\right\rangle+2\left\langle \nabla\varphi'(nz),{\rm Ad}s\nabla^{\prime}\psi'(nz)\right\rangle+\nonumber \\
+\left\langle {\rm Ad}(nz)\nabla^{\prime
}\varphi'(nz),\nabla\psi'(nz)\right\rangle-\left\langle {\nabla\varphi'(nz),\rm Ad}(nz)\nabla^{\prime
}\psi'(nz)\right\rangle. \nonumber
\end{eqnarray*}

Note that the last formula differs from formula (\ref{tau}) for the Poisson bracket on the Poisson--Lie group $Z_e$ by the last four terms.

\end{document}